\documentclass[11pt,twoside,english]{article}
\usepackage{mathptmx}
\usepackage{graphicx}
\usepackage[english,french]{babel}
\usepackage{blindtext}

\usepackage[T1]{fontenc}
\usepackage[utf8]{inputenc}
\usepackage[a4paper]{geometry}
\usepackage{color}
\usepackage{babel}
\usepackage{mathrsfs}
\usepackage{amsmath}
\usepackage{amsthm}
\usepackage{amssymb}
\usepackage{esint}
\usepackage[unicode=true,
 bookmarks=true,bookmarksnumbered=true,bookmarksopen=false,
 breaklinks=false,pdfborder={0 0 1},backref=false,colorlinks=true]
 {hyperref}
\hypersetup{
 pdfauthor={G Xi}}

\makeatletter

\newtheorem{Theorem}{Theorem}[section]
\newtheorem{Definition}[Theorem]{Definition}

\newtheorem{Lemma}[Theorem]{Lemma}
\newtheorem{Corollary}[Theorem]{Corollary}
\newtheorem{Remark}[Theorem]{Remark}

\newtheorem{Hypothesis}{Hypothesis}

\numberwithin{equation}{section}
\theoremstyle{plain}
\newtheorem{thm}{\protect\theoremname}[section]
\theoremstyle{plain}
\newtheorem{prop}[thm]{\protect\propositionname}
\theoremstyle{plain}
\newtheorem{lem}[thm]{\protect\lemmaname}
\theoremstyle{remark}

\theoremstyle{plain}
\newtheorem{cor}[thm]{\protect\corollaryname}
\theoremstyle{remark}
\newtheorem*{rem*}{\protect\remarkname}

\def\r2{\mathbb{R}^2}
\def\le{\left}
\def\r{\right}

\def\a{\alpha}
\def\d{\delta}

\def\e{\epsilon}
\def\la{\lambda}
\def\si{\sigma}

\def\H{{\mathcal H}}

\def\s0t{\sup_{t \in [0,T]}}

\def\ds{\displaystyle}

\def\beq{\begin{equation}}
\def\eeq{\end{equation}}
\def\barr{\begin{array}}
\def\earr{\end{array}}
\def\vs{\vspace{.01mm}   \\}
\def\rd{\reals\,^{d}}

\newcommand{\nat}{\mathbb N}
\newcommand{\E}{\mathbb E}

\newcommand{\reals}{\mathbb R}

\newcommand{\norm}[1]{\left\| #1\right\|}

\newcommand{\inner}[1]{\langle #1\rangle}
\newcommand{\Inner}[1]{\big\langle #1\big\rangle}

\newcommand{\abs}[1]{\lvert #1\rvert}
\def\H{{\mathcal{H}}}

\def\E{{\mathbb{E}}}

\usepackage{babel}
\allowdisplaybreaks[4]

\date{}

  \providecommand{\corollaryname}{Corollary}
  
  \providecommand{\lemmaname}{Lemma}
\providecommand{\theoremname}{Theorem}

\theoremstyle{plain}

\makeatother

\providecommand{\corollaryname}{Corollary}
\providecommand{\lemmaname}{Lemma}
\providecommand{\propositionname}{Proposition}
\providecommand{\remarkname}{Remark}
\providecommand{\theoremname}{Theorem}

\begin{document}
\global\long\def\divg{{\rm div}\,}%

\global\long\def\curl{{\rm curl}\,}%

\global\long\def\rt{\mathbb{R}^{3}}%

\global\long\def\rd{\mathbb{R}^{d}}%

\global\long\def\rtwo{\mathbb{R}^{2}}%

\global\long\def\e{\epsilon}%

\title{On the small noise limit in the Smoluchowski-Kramers approximation of nonlinear wave equations with variable friction}
\author{Sandra Cerrai\thanks{Department of Mathematics, University of Maryland, College Park, MD
20742, USA. Emails: cerrai@umd.edu, mxie2019@umd.edu}\ \,\thanks{Partially supported by NSF grants DMS-1712934 -  {\em Analysis of Stochastic Partial Differential Equations with Multiple Scales}  and  DMS-1954299 - {\em 
 Multiscale Analysis of Infinite-Dimensional Stochastic Systems}} \ and Mengzi Xie\footnotemark[1]}
\maketitle

\selectlanguage{english}

\date{}

\maketitle

\begin{abstract}
We study the validity of a large deviation principle for a class of  stochastic  nonlinear damped wave equations, of Klein-Gordon type, in the joint small mass and small noise limit. The friction term is assumed to be state dependent.

\vspace{.3cm}

\end{abstract}

\section{Introduction}
\label{sec1}

In this article we deal with this class of stochastic wave equations with state-dependent damping on a bounded smooth domain $\mathcal{O}\subset \mathbb{R}^d$
\begin{equation}
\label{SPDE1}
\le\{\begin{array}{l}
\ds{\mu\partial_t^2 u_\mu(t,x)=\Delta u_\mu(t,x)- \gamma (u_\mu(t,x)) \partial_t u_\mu(t,x) + f(x,u_\mu(t,x))+ \si(u_\mu(t,\cdot))\partial_t w^Q(t,x),}\\[10pt]
\ds{u_\mu(0,x)=u_0(x),\ \ \ \ \partial_t u_\mu(0,x)=v_0(x),\ \ \ \ \ \ \  u_{\mu}(t,x)=0,\ \ x \in\,\partial \mathcal{O},}
\end{array}\r.
\end{equation}
depending on a parameter $0<\mu<<1$. Here the friction coefficient $\gamma$ is  strictly positive and bounded  and the nonlinearity $f$ is  either a Lipschitz-continuous  function (in this  case we can consider any $d\geq 1$) or a locally Lipschitz-continuous function  of the Klein-Gordon type (in this case we can only take $d=1$). The noise $w^Q(t)$ is a cylindrical $Q$-Wiener process and $\si$ is a suitable Lipschitz-continuous operator-valued function.

The solution $u_\mu$ of equation \eqref{SPDE1} can be seen  as the displacement field of some particles in a domain $\mathcal{O}$, subject to interaction forces represented by  the Laplacian and  to nonlinear reactions represented by $f$, in the presence of  a random external forcing $\si(u_\mu(t,\cdot))\partial_t w^Q(t)$ and a state-dependent friction $\gamma (u_\mu(t)) \partial_t u_\mu(t)$.  A series of papers has investigated the validity of the so-called Smoluchowski-Kramers approximation, that describes the limiting behavior of the solution $u_\mu$, as the  density $\mu$  of the particles vanishes. 
For the finite dimensional case, the existing literature is quite broad and  we refer in particular to \cite{f}, \cite{fh}, \cite{hhv}, \cite{hmdvw} and  \cite{spi} (see also \cite{CF3}, \cite{CWZ} and \cite{lee} for systems subject to a magnetic field and \cite{hu} and \cite{Nguyen} for some related multiscaling problems). 

In recent years there has been an intense activity dealing with the Smoluchowski-Kramers approximation of   infinite dimensional systems. To this purpose, we refer to \cite{CF1}, \cite{CF2}, \cite{salins} and \cite{Lv2} for the  case of constant damping term (see also \cite{CS3} where systems subject to a magnetic field are studied), and to \cite{CX} for the case of  state-dependent damping. As a matter of fact, these two situations are quite different. When $\gamma$ is constant,  $u_\mu$ converges to the solution of the stochastic parabolic problem  
\begin{equation}
\label{SPDE2}
\le\{\begin{array}{l}
\ds{\gamma\,\partial_t u(t,x)=\Delta u(t,x) + f(x,u(t,x))+ \si(u(t,\cdot))\partial_t w^Q(t,x),}\\[10pt]
\ds{u(0,x)=u_0(x),\ \ \ \ \partial_t u(0,x)=v_0(x),\ \ \ \ \ \ \  u_{\mu}(t,x)=0,\ \ x \in\,\partial \mathcal{O}.}
\end{array}\r.
\end{equation}
However, when $\gamma$ is not constant, because of the interplay between the state-dependent friction and  the noise, an extra drift is created and in \cite{CX} it has been proven that the  limiting equation becomes
\begin{equation}
\label{SPDE3}
\le\{\begin{array}{l}
\ds{\gamma(u(t,x))\partial_t u(t,x)= \Delta u(t,x) +f(u(t,x)) -\frac{\gamma'(u(t,x))}{2\gamma^2(u(t,x))} \sum_{i=1}^\infty |[\si(u(t,\cdot)Qe_i](x)|^2 +\si(u(t,\cdot))\partial_t w^Q(t,x),}\\[18pt]
\ds{u(0,x)=u_0(x), \ \ \ \ \ \ \  u(t)_{|_{\partial\mathcal{O}}}=0,}
\end{array}\r.
\end{equation}
where $\{Qe_i\}_{i \in\,\mathbb{N}}$ is a complete orthonormal basis of the reproducing kernel of the noise.

Once proved the validity of the small mass limit, it is important to understand how stable such an approximation is with respect to other important asymptotic features of the two systems, such as for example the long time behavior.  To this purpose,  in \cite{CGH} and \cite{CF1} it is shown that  the statistically invariant states of equation \eqref{SPDE1} (in case of constant friction) converge in a suitable sense to the invariant measure of  equation \eqref{SPDE2}. In the same spirit, the two papers  \cite{sal} and \cite{sal2} are devoted to an analysis of the convergence of the quasi-potential, that describes, as known, the asymptotics of the exit times and the large deviation principle for the invariant measure.

In the present paper we are interested in studying the validity of a large deviation principle for the following equation
\begin{equation}
\label{SPDE1-bis}
\le\{\begin{array}{l}
\ds{\mu\partial_t^2 u_\mu(t,x)=\Delta u_\mu(t,x)- \gamma (u_\mu(t,x)) \partial_t u_\mu(t,x) + f(x,u_\mu(t,x))+ \sqrt{\mu}\,\si(u_\mu(t,\cdot))\partial_t w^Q(t,x),}\\[10pt]
\ds{u_\mu(0,x)=u_0(x),\ \ \ \ \partial_t u_\mu(0,x)=v_0(x),\ \ \ \ \ \ \  u_{\mu}(t,x)=0,\ \ x \in\,\partial \mathcal{O},}
\end{array}\r.
\end{equation}
where, together with the mass, we are also assuming that  the intensity of the noise vanishes.
Our aim is proving that in the joint small mass and small noise limit the family of random variables $\{u_\mu\}_{\mu>0}$  satisfy a large deviation principle in the space $C([0,T];L^p(\mathcal{O}))$, (for some $p>2$ depending on the dimension $d$),  with respect to the action functional
\[I_T(u)=\frac 12\,\left\{\int_0^T \Vert \varphi(t)\Vert_H^2\,dt\ :\ u(t)=u^\varphi(t),\ t \in\,[0,T]\right\},\]
	where $u^\varphi(t)$ denotes the solution of  the controlled quasi-linear parabolic equation
	\begin{equation}
	\label{quasi}
	\left\{ \begin{array}{l}
 \ds{\gamma(u(t,x))\partial_t u(t,x)=\Delta u(t,x)+	f(x,u(t,x))+\sigma(u(t,\cdot))\varphi(t,x),}\\
 \vs
 \ds{u(0,x)=u_0(x),\ \ \ \ u(t,x)=0, \ \ \ x \in\,\partial \mathcal{O}.}
 \end{array}\right.
 \end{equation}

 This means in particular that, in spite of the fact  that in the presence of a non-constant friction coefficient the Smoluchowski-Kramers approximation of equation \eqref{SPDE1} leads to equation \eqref{SPDE3}, the large deviation principle is consistent with equation \eqref{SPDE2}. 
 
 Due to the nature of our problem, the weak-convergence approach to large deviation, as developed  in \cite{bdm} for SPDEs, is the ideal tool for our proof. As known, such an approach  requires a thorough  
analysis of the following controlled version of equation \eqref{SPDE1-bis}
\begin{equation}
	\label{fine}
\le\{\begin{array}{l}
\ds{\mu\partial_t^2 u_\mu(t,x)=\Delta u_\mu(t,x)- \gamma (u_\mu(t,x)) \partial_t u_\mu(t,x) + f(x,u_\mu(t,x))+\si(u_\mu(t,\cdot))Q\varphi(t,x)}\\
\vs
\ds{\quad\quad\quad\quad\quad\quad+ \sqrt{\mu}\,\si(u_\mu(t,\cdot))\partial_t w^Q(t,x),\ \ \ t>0,\ \ \ \ \ x \in\,\mathcal{O},}\\[10pt]
\ds{u_\mu(0,x)=u_0(x),\ \ \ \ \partial_t u_\mu(0,x)=v_0(x),\ \ \ \ \ \ \  u_{\mu}(t,x)=0,\ \ \ x \in\,\partial \mathcal{O},}
\end{array}\r.\end{equation}

 In \cite{CX}, only the case of Lipschitz $f$ and  bounded $\si$ is considered. However, in relevant models it is important to consider  non-Lipschitz nonlinearities. For this reason,   in this paper we are  considering also nonlinearities $f$ having polynomial growth and satisfying some monotonicity conditions. We would like to stress that this, together with the fact that we allow 
 the diffusion coefficients $\si$ to have  linear growth and  we have to add a control in equation \eqref{fine}, has required the introduction of  new arguments, compared with \cite{CX}, already in the proof of the well-posedness result. 

After we have shown that equation \eqref{fine} admits a unique solution $u_\mu^\varphi$, for every fixed $\mu>0$ and for every predictable control,   we have proven suitable a priori bounds for such solution and its time derivative. Then, we have introduced 
$\rho_\mu:=g(u_\mu^{\varphi_\mu})$, where $g^\prime=\gamma$ and $\{\varphi_\mu\}_{\mu>0}$ is a family of controls all contained $\mathbb{P}$-a.s. in a ball of $L^2(0,T;L^2(\mathcal{O}))$,  and we have shown that these estimates   imply the tightness of the family $\{\rho_\mu\}_{\mu \in\,(0,\mu_T)}$ in $C([0,T];H^{\d})$, for some $\mu_T>0$ and for every $\d<1$.

Next, we have shown how,  for every sequence $\{\mu_k\}_{k \in\,\mathbb{N}}$ converging to zero, every limit point $\rho$ of $\{\rho_{\mu_k}\}_{k \in\,\mathbb{N}}$ is a weak solution of the  
deterministic controlled problem
\[\left\{\begin{array}{l}
\ds{\partial_t \rho(t,x)=\text{div}\,[b(\rho(t,x))]+f_g(x,\rho(t,x))+\sigma_g(\rho(t,\cdot))\varphi(t,x),\ \ \ \ t>0,	\ \ \ \ \ x \in\,\mathcal{O},}\\
\vs\ds{\rho(0,x)=g(u_0(x)),\ \ \ \ \ \ \ \ \ \rho(t,x)=0,\ \ \ x \in\,\partial \mathcal{O},}
\end{array}\right.\]
where $b=1/\gamma \circ g^{-1}$, $f_g=f\circ g^{-1}$, and $\si_g=\si\circ g^{-1}$. 
In order to identify uniquely the limit point and prove that $\{\rho_{\mu_k}\}_{k \in\,\mathbb{N}}$ converges to $\rho$, we had to extend some results proved in \cite{liu} about  nonlinear evolution equations with locally monotone coefficients  to prove that the equation above has a unique solution.
Then, by defining $u:=g^{-1}(\rho)$, we have obtained the convergence of $u_\mu^{\varphi_\mu}$ to the solution of the controlled  equation \eqref{fine} and this has allowed us to conclude our proof.

Finally, we would like to mention that in Appendix \ref{A1} we have extended the results of \cite{CX} and provided a proof of the validity of the Smoluchowski-Kramers approximation for quasi-monotone $f$ having polynomial growth and unbounded diffusion $\sigma$ (see Hypothesis \ref{Hypothesis3}). This has required the proof of quite non-trivial a-priori bounds for the solution $u_\mu$ and its time derivative $\partial_t u_\mu$, and the introduction of suitable functional spaces where tightness holds and the small-mass limit can be proven.

\section{Notations and assumptions}
\label{sec2}
Throughout the present paper $\mathcal{O}$ is a bounded domain in $\mathbb{R}^d$, with smooth boundary. We denote by $H$ the Hilbert space $L^2(\mathcal{O})$ and by $\langle \cdot , \cdot \rangle_H$ the corresponding  inner product.  $H^1$ is the completion of $C_0^\infty(\mathcal{O})$ with respect to norm 
\[
\Vert u \Vert_{H^1}^2:=\Vert \nabla u \Vert_H^2=\int_\mathcal{O} \vert \nabla u(x)\vert^2 dx,
\]
and $H^{-1}$ is the dual space to $H^1$. Then $H^1$, $H$ and $H^{-1}$ are all complete separable metric spaces, and  $H^1\subset H\subset H^{-1}$,
with compact embeddings. In what follows,  we shall denote  
\begin{equation*}
    \mathcal{H}=H\times H^{-1},\ \ \ \ \  \mathcal{H}_1=H^1 \times H. 
\end{equation*}

Given the domain $\mathcal{O}$, we denote by $\{e_i\}_{i\in\mathbb{N}}\subset H^1$ the complete orthonormal basis of $H$ which diagonalizes the Laplacian $\Delta$, endowed with Dirichlet boundary conditions on $\partial\mathcal{O}$. Moreover, we denote by $\{-\alpha_i\}_{i\in \mathbb{N}}$ the corresponding sequence of eigenvalues, i.e.
 \[\Delta e_i=-\alpha_i e_i,\ \ \ \  i\in \mathbb{N}.\]
 % Given 
 % \[u=\sum_{i=1}^\infty b_i e_i,\ \ \ \  v=\sum_{i=1}^\infty c_i e_i,\]
 % for some sequences of real numbers $(b_i)_{i\in\,\mathbb{N}}$ and $(c_i)_{i\in\,\mathbb{N}}$,
%we have
%\begin{equation}\label{InnerProducts}
  %  \langle u,v \rangle_{H^1}=\sum_{i=1}^\infty \alpha_i b_i c_i,\quad \langle u,v \rangle_{H}=\sum_{i=1}^\infty b_i c_i,\quad \langle u,v\rangle_{H^{-1}}=\sum_{i=1}^\infty \frac{1}{\alpha_i} b_i c_i.
%\end{equation}
%From \eqref{InnerProducts} we can derive the Poincar\'e inequality
%\begin{equation}\label{Poincare}
   % \Vert u\Vert _{H}\leq \frac{1}{\sqrt{\alpha_1}}\Vert u\Vert _{H^1},\ \ \ \ u\in H^1,\ \ \ \ \  \Vert u\Vert _{H^{-1}}\leq \frac{1}{\sqrt{\alpha_1}} \Vert u\Vert _H,\ \ \ \ u\in H.
%\end{equation}

\subsection{The stochastic term}

We assume that $w^Q(t)$ is a cylindrical $Q$-Wiener process, defined on a complete stochastic basis  $(\Omega,\mathcal{F},(\mathcal{F}_t)_{t\geq 0},\mathbb{P})$. This means that $w^Q(t)$ can be formally written as 
\[
w^Q(t)=\sum_{i=1}^\infty Q e_i \beta_i(t),
\]
where $\{\beta_i\}_{i\in \mathbb{N}}$ is a sequence of independent standard Brownian motions on $(\Omega,\mathcal{F},(\mathcal{F}_t)_{t\geq 0},\mathbb{P})$,   $\{e_i\}_{i\in\,\mathbb{N}}$ is the complete orthonormal system introduced above that diagonalizes the Laplace operator, endowed with Dirichlet boundary conditions, and $Q:H\rightarrow H$ is a bounded  linear operator,. 
When $Q=I$, $w^I(t)$ will be denoted by $w(t)$. In particular, we have $w^Q(t)=Qw(t)$.

In what follows we shall denote by $H_Q$ the set $Q(H)$. $H_Q$ is the reproducing kernel of the noise $w^Q$ and  is a Hilbert space, endowed with the inner product
\[\langle Qh, Qk\rangle_{H_Q}=\langle h,k\rangle_H,\ \ \ \ h, k \in\,H.\] Notice that the sequence $\{Q e_i\}_{i \in\,\mathbb{N}}$ is a complete orthonormal system in $H_Q$. Moreover, if $U$ is any Hilbert space containing $H_Q$ such that the embedding  of $H_Q$ into $U$ is Hilbert-Schmidt, we have that 
\begin{equation}
  \label{contb}w^Q \in\,C([0,T];U).
\end{equation}

Next, we recall that for every two separable Hilbert spaces $E$ and $F$, $\mathcal{L}_2(E,F)$ denotes the space of Hilbert-Schmidt operators from $E$ into $F$. $\mathcal{L}_2(E,F)$ is a Hilbert space, endowed with the inner product
\[\langle A,B\rangle_{\mathcal{L}_2(E,F)}=\mbox{Tr}_E\,[A^\star B]=\mbox{Tr}_F[B A^\star].\]
As well known, $\mathcal{L}_2(E,F) \subset \mathcal{L}(E,F)$ and
\begin{equation}
\label{sm5}
\Vert A\Vert_{\mathcal{L}(E,F)}\leq 	\Vert A\Vert_{\mathcal{L}_2(E,F)}.
\end{equation}

\begin{Hypothesis}
\label{Hypothesis1}
The mapping $\si:H\to \mathcal{L}_2(H_Q,H)$ is defined by 
\[[\sigma(h)Qe_i](x) = \sigma_i(x,h(x)), \ \ \ \ x \in\,\mathcal{O}\ \ \ \ \ h \in\,H,\ \ \ \ i \in\,\mathbb{N},\] 
for some mapping $\sigma_i:\mathcal{O}\times \mathbb{R}\rightarrow \mathbb{R}$. We assume that there exists $L>0$ such that
\begin{equation}
\label{sgfine1}	
\sup_{x \in\,\mathcal{O}}\, \sum_{i=1}^\infty \vert \sigma_i(x,y_1) - \sigma_i(x,y_2)\vert^2 \leq L\,\vert y_1-y_2\vert^2,\ \ \ \ \ y_1, y_2 \in\,\mathbb{R}.
\end{equation}
Moreover,
\begin{equation}
\label{sm1}
\sup_{x \in\,\mathcal{O}}\, \sum_{i=1}^\infty \vert \sigma_i(x,0)\vert^2=:\si_0^2<\infty.
    \end{equation}

\end{Hypothesis}

\begin{Remark}
{\em 
\begin{enumerate}
\item Condition \eqref{sgfine1} implies that $\si$ is Lipschitz continuous. Namely
 for any $h_1,h_2\in H$
\begin{equation}
\label{sm2}
    \Vert\si(h_1)-\si(h_2)\Vert_{\mathcal{L}_2(H_Q,H)} \leq \sqrt{L}\, \Vert h_1 -h_2\Vert_H.
\end{equation}
This, together with condition \eqref{sm1}, implies also that $\si$ has linear growth,  that is
\begin{equation}
\label{sm4}
\Vert\si(h)\Vert_{\mathcal{L}_2(H_Q,H)} \leq \sqrt{L}\, \Vert h\Vert_H+|\mathcal{O}|^{1/2}\si_0.
\end{equation}
\item If $\si$ is constant,  then Hypothesis \ref{Hypothesis1} means that $\si Q$ is a Hilbert-Schmidt operator in $H$. \item If $\si$ is not constant,  Hypothesis \ref{Hypothesis1} is satisfied for example when
\[ [\si(h)Qk](x)=s(x,h(x))Qk(x),\ \ \ \ x \in\,\mathcal{O}, \ \ \ \ h, k \in\,H,\]
for some measurable function $s:\mathcal{O}\times \mathbb{R}\to \mathbb{R}$  such that $s(x,\cdot):\mathbb{R}\to \mathbb{R}$ is Lipschitz continuous, uniformly with respect to $x \in\,\mathcal{O}$, and for some $Q \in\,\mathcal{L}(H)$ such that
\begin{equation}
\label{sm6}
\sum_{i=1}^\infty \Vert Qe_i\Vert^2_{L^\infty(\mathcal{O})}<\infty.	
\end{equation}
In case $Q$ is diagonalizable with respect the basis $(e_i)_{i \in\,\mathbb{N}}$, with $Q e_i=\la_i e_i$,
condition \eqref{sm6} reads
\begin{equation}\label{gx005}
    \sum_{i=1}^\infty \la_i^2\Vert e_i\Vert^2_{L^\infty(\mathcal{O})}<\infty.
\end{equation}
In general (see \cite{greiser}), we have 
\[\Vert e_i\Vert_{L^\infty(\mathcal{O})} \leq c\, i^\alpha,\]
for some $\alpha>0$, and \eqref{gx005} becomes
\[\sum_{i=1}^\infty \la_i^2 \, i^{2\alpha}<\infty.\]
In particular, when $d=1$ or the domain is a hyperrectangle when $d>1$ the eigenfunctions $(e_i)_{i\in\mathbb{N}}$ are equi-bounded and \eqref{gx005} becomes $Q \in\,\mathcal{L}_2(H).$
\end{enumerate}
}
\end{Remark}

\subsection{The coefficients $\gamma$ and $f$}
Throughout the paper, we shall assume that the friction coefficient satisfies the following condition

\begin{Hypothesis}\label{Hypothesis2}
The mapping $\gamma$ belongs to $C^1_b(\mathbb{R})$ and there exist $\gamma_0$ and $\gamma_1$ such that
\begin{equation}
\label{nonlinearity assumption}
0<\gamma_0\leq \gamma( r)\leq \gamma_1,\ \ \ \ \ \  r\in\mathbb{R}.
\end{equation}

\end{Hypothesis}

In what follows, we shall define
\[g(r)=\int_0^r \gamma(\si)\,d\si,\ \ \ \ \ r \in\,\mathbb{R}.\]
\begin{Remark}
{\em \begin{enumerate}
\item Clearly $g(0)=0$ and $g^\prime(r)=\gamma(r)$. In particular, due to  \eqref{nonlinearity assumption},  $g$ is uniformly Lipschitz continuous on $\mathbb{R}$. 

\item The  function $g$ is strictly increasing and
\[
\le(g(r_1)-g(r_2)\r)(r_1-r_2)\geq \gamma_0\,|r_1-r_2|^2,\ \ \ \ \ r_1, r_2 \in\,\mathbb{R}.
\]

 \end{enumerate}
}
\end{Remark}

\bigskip

As far as the nonlinearity $f$ is concerned, in this paper we shall consider two situations: $f$ is Lipschitz continuous and $\mathcal{O}$ is a bounded smooth domain in $\mathbb{R}^d$, for any arbitrary $d\geq 1$, or $f$ is only locally Lipschitz continuous with  polynomial growth and $\mathcal{O}$ is a bounded interval in $\mathbb{R}$.	
 
 \begin{Hypothesis}
\label{Hypothesis3-bis}
The mapping $f:\mathcal{O}\times \mathbb{R}\to \mathbb{R}$ is measurable and there exists $c>0$ such that
\[\sup_{x \in\,\mathcal{O}}|f(x,r)-f(x,s)|\leq c\,|r-s|,\ \ \ \ r, s \in\,\mathbb{R}.\]
Moreover 
\[\sup_{x \in\,\mathcal{O}}|f(x,0)|<\infty.\]
\end{Hypothesis}
In what follows, for every function $u:\mathcal{O}\to\reals$, we shall denote
\[F(u)(x)=f(x,u(x)),\ \ \ \ \ x \in\,\mathcal{O}.\]

\begin{Hypothesis}
\label{Hypothesis3}
We have $\mathcal{O}=[0,L]$ and the mapping $f:[0,L]\times\mathbb{R}\to\mathbb{R}$ is  measurable and satisfies the following conditions.
\begin{enumerate}
	\item There exist $\theta>1$ and  $c_1>0$ such that for every $r \in\,\mathbb{R}$
\begin{equation}
\label{sm9}
\sup_{x \in\,[0,L]}|f(x,r)|\leq c_1\left(1+|r|^\theta\right),\ \ \ \ \ \sup_{x \in\,[0,L]}|\partial_r f(x,r)|\leq c_1\left(1+|r|^{\theta-1}\right).
\end{equation}

\smallskip

Moreover, there exists  $c_2>0$ such that for every $r \in\,\mathbb{R}$ and $x \in\,[0,L]$
\begin{equation}
\label{sm12}
\mathfrak{f}(x,r):=\int_0^r f(x,s)\,ds\leq c_2 \left(1-|r|^{\theta+1}\right).\end{equation}
\item For every $x \in\,[0,L]$, the function $f(x,\cdot):\mathbb{R}\to\reals$ is differentiable and 
\begin{equation}
\label{sm11}
\sup_{(x,r) \in\,[0,L]\times \reals}\partial_r f(x,r)\leq 0.
	\end{equation}
\item	For every $r \in\,\mathbb{R}$, the function $f(\cdot,r):[0,L]\to\reals$ is differentiable and 
\[
	\sup_{x \in\,[0,L]}\left|\partial_x f(x,r)\right|\leq c\left(1+|r|\right),\ \ \ r \in\,\mathbb{R}.
\] 

\end{enumerate}

\end{Hypothesis}

%In what follows, for every function $u:\mathcal{O}\to\mathbb{R}$ we shall define
%\[F(u)(x)=f(u(x)),\ \ \ \ x \in\,\mathcal{O}.\]

\begin{Remark}
	\label{sm8}
	{\em   \begin{enumerate}
   \item A typical example of a function $f$ satisfying Hypothesis \ref{Hypothesis3} is
   \[f(r)=-a\,|r|^{\theta-1}r.\]

%   \item In fact, our result holds even in the more general case condition \eqref{sm11} is replaced by the more general condition
 %  \[
%\left(f(r)-f(s)\right)(r-s)\leq c\,\left(r-s\right)^2,\ \ \ \ r,s \in\,\mathbb{R},
	%\]in
	%for some $c \in\,\mathbb{R}$. Here we are assuming \eqref{sm11} just for simplicity of presentation.
	
   \item When $d=1$, we have that $H^1\hookrightarrow L^\infty(\mathcal{O})$, and then $F(u)\in\,H$, for every $u \in\,H^1$.
   \item We are assuming \eqref{sm11} just for the sake of simplicity. In fact, our results remain true under the condition
   \[\sup_{(x,r) \in\,[0,L]\times \reals}\partial_r f(x,r)<\infty.\]
   
   \item From \eqref{sm12} and \eqref{sm11}, it is not hard to show that for every $r\in\mathbb{R}$
   \begin{equation}
   	\label{sm145}
   	\sup_{x\in[0,L]}rf(x,r)\leq c_{2}\Big(1-\abs{r}^{\theta+1}\Big).
   \end{equation}
   
   Indeed, if we consider the function 
   \begin{equation*}
   	G(x,r):=\mathfrak{f}(x,r)-rf(x,r),\ \ \ r\in\mathbb{R},\ x\in[0,L],
   \end{equation*}
	
	then for every $x\in[0,L]$, $\partial_{r}G(x,r)=-r\partial_{r}f(x,r)\geq0$ if $r>0$, and $\partial_{r}G(x,r)\leq0$ if $r<0$. Note that $G(x,0)=0$, we have $G(x,r)\geq0$, and thus \eqref{sm145} follows from \eqref{sm12}.

 \item Thanks to \eqref{sm11} we have
 \[
 \langle F(u)-F(v),u-v\rangle_H\leq 0,\ \ \ \ \ u, v \in\,H^1.
 \]
 In particular, there exists some $c>0$ such that
 \begin{equation}
 \label{sm100}
 \langle F(u),u\rangle_H\leq c\,\Vert u\Vert_{H},\ \ \ \  \ u \in\,H^1.	
 \end{equation}

\item Due to \eqref{sm9}, for every $u, v \in\,H^1$, we have
\[\Vert F(u)-F(v)\Vert_H^2 \leq c \int_\mathcal{O}\left(1+|u(x)|^{2(\theta-1)}+|v(x)|^{2(\theta-1)}\right)|u(x)-v(x)|^2\,dx,
\]
so that 
\begin{equation}
\label{sm15}
\begin{array}{ll}
\ds{\Vert F(u)-F(v)\Vert_H} & \ds{\leq c\left(1+\Vert u\Vert _{H^1}^{\theta-1}+	\Vert v\Vert _{H^1}^{\theta-1}\right)\Vert u-v\Vert_{H}.}
\end{array}
\end{equation}
In particular, we have
\[
 \Vert F(u) \Vert_H\leq c\left(1+\Vert u\Vert_{H^1}^\theta\right),\ \ \ \ u \in\,H^1.
 \]
 \item In the same way
\begin{equation}
\label{sm15-l1}
\begin{array}{ll}
\ds{\Vert F(u)-F(v)\Vert_{L^1}\leq }  &  \ds{c \int_\mathcal{O}\left(1+|u(x)|^{\theta-1}+|v(x)|^{\theta-1}\right)|u(x)-v(x)|\,dx}\\
& \vs
& \ds{\leq c\,\left(1+\Vert u\Vert_{L^{2(\theta-1)}}^{\theta-1}+\Vert v\Vert_{L^{2(\theta-1)}}^{\theta-1}\right)\Vert u-v\Vert_H.}
\end{array}
\end{equation}

   \end{enumerate}}
\end{Remark}

Now, by proceeding as in \cite{mm} (see also \cite{CF2}), for every $n \in\,\mathbb{N}$ and $x \in\,\mathcal{O}$ we define
\begin{equation}
\label{sm300}
f_{n}(x,r):=\begin{cases}
\ds{f(x,n)+(r-n)\partial_r f(x,n),}	&  \ds{\text{if}\ r\geq n,}\\
& \vs
\ds{f(x,r),} & \ds{\text{if}\ r \in\,[-n,n],}\\
&\vs
\ds{f(x,-n)+(r+n)\partial_r f(x,-n),}  &  \ds{\text{if}\ r\leq -n.}
\end{cases}	
\end{equation}
Clearly, for every $n \in\,\mathbb{N}$, the mapping $f_n(x,\cdot):\reals\to \reals$ is Lipschitz continuous, uniformly with respect to $x \in\,[0,L]$, and 
\[f_n(x,r)=f(x,r),\ \ \ \ x \in\,[0,L],\ \ r \in\,[-n,n].\] 
Moreover,
\begin{equation}
\label{sm9-n}
\sup_{x \in\,[0,L]}|f_n(x,r)|\leq c\left(1+|r|^\theta\right),\ \ \ \ \ \sup_{x \in\,[0,L]}|\partial_r f_n(x,r)|\leq c\left(1+|r|^{\theta-1}\right),
\end{equation}
for some constant $c$ independent of $n$, and there exists $n_0 \in\,\mathbb{N}$ such that for every $n\geq n_0$
\begin{equation}
\label{sm83}	\mathfrak{f}_n(x,r):=\int_0^r f_n(x,s)\,ds\leq c\left(1-n^{\theta+1}\right),\ \ \ \ r \in\,\reals,\ \ \ \ x \in\,[0,L].
\end{equation}
.

\section{The problem and the method}
\label{sec3}

As we mentioned in the introduction, we are interested in the study of the validity of a large deviation principle, as $\mu\downarrow 0$ for the family $\{\mathcal{L}(u_\mu)\}_{\mu >0}$, where $u_\mu$ is the solution of  equation \eqref{SPDE1}. Our final goal is proving the following result.

\begin{Theorem}
	\label{mainteo}
	Assume Hypotheses \ref{Hypothesis1} and \ref{Hypothesis2} and either Hypothesis \ref{Hypothesis3-bis} or Hypothesis \ref{Hypothesis3} and fix $p<\infty$, if $d=1,2$, and $p<2d/(d-2)$, if $d>2$. Then, for every $(u_0,v_0) \in\,\mathcal{H}_1$ and  $T>0$, the family $\{\mathcal{L}(u_\mu)\}_{\mu >0}$ satisfly a large deviation principle in $C([0,T];L^p(\mathcal{O}))$, as $\mu\downarrow 0$, with action functional
	\begin{equation}
	\label{sm17}	
	I_T(u)=\frac 12\,\left\{\int_0^T \Vert \varphi(t)\Vert_H^2\,dt\ :\ u(t)=u^\varphi(t),\ t \in\,[0,T]\right\},
	\end{equation}
	where $u^\varphi(t)$ denotes the unique weak solution to the quasi-linear parabolic equation
	\begin{equation}
	\label{sm18}
	\left\{ \begin{array}{l}
 \ds{\partial_t u(t,x)=\gamma^{-1}(u(t,x))\left[\Delta u(t,x)+	f(x,u(t,x))+\sigma(u(t,\cdot))\varphi(t,x)\r],}\\
 \vs\ds{u(0,x)=u_0(x),\ \ \ \ u(t,x)=0, \ \ \ x \in\,\partial \mathcal{O}.}
 \end{array}\right.
	\end{equation}

\end{Theorem}

Theorem \ref{mainteo} is proved by following  the classical weak convergence approach to large deviations, as developed for SPDEs in \cite{bdm}. To this purpose, we need first to introduce  some notations. For every $T>0$, we denote by $\mathcal{P}_T$ the set of predictable processes in $L^2(\Omega\times [0,T];H)$, and for every $M>0$ we introduce the sets
\[\mathcal{S}_{T,M}:=\left\{ \varphi \in\,L^2_w(0,T;H)\ :\ \Vert \varphi\Vert_{L^2(0,T;H)}\leq M\right\},\]
and
\[\Lambda_{T,M}:=\left\{ \varphi \in\,\mathcal{P}_T\ :\ \varphi \in \mathcal{S}_{T,M},\ \mathbb{P}-\text{a.s.}\right\}.\]

Next, for every $\varphi \in\,\mathcal{P}_T$ we consider the controlled version of equation \eqref{SPDE1}
\begin{equation}
\label{SPDE-controlled}
\le\{\begin{array}{l}
\ds{\mu\partial_t^2 u_\mu(t,x)=\Delta u_\mu(t,x)- \gamma (u_\mu(t,x)) \partial_t u_\mu(t,x) + f(x,u_\mu(t,x))+\si(u_\mu(t,\cdot))Q\varphi(t,x)}\\
\vs
\ds{\quad\quad\quad\quad\quad\quad+ \sqrt{\mu}\,\si(u_\mu(t,\cdot))\partial_t w^Q(t,x),\ \ \ t>0,\ \ \ \ \ x \in\,\mathcal{O},}\\[10pt]
\ds{u_\mu(0,x)=u_0(x),\ \ \ \ \partial_t u_\mu(0,x)=v_0(x),\ \ \ \ \ \ \  u_{\mu}(t,x)=0,\ \ \ x \in\,\partial \mathcal{O},}
\end{array}\r.
\end{equation}
The well-posedness of the equation above has been proven in \cite{CX} in the case the nonlinearity $f$ is Lipschitz-continuous, the diffusion coefficient $\si$ is bounded and the control $\varphi=0$.
In what follows, we will prove that also under the more general conditions we are assuming for $f$ and $\si$, the following results holds.

\begin{Theorem}
\label{well-posedness}
Under Hypotheses \ref{Hypothesis1} and \ref{Hypothesis2} and either Hypothesis \ref{Hypothesis3-bis} or Hypothesis \ref{Hypothesis3}, for every $T, M>0$ and $\varphi \in\,\Lambda_{T,M}$ and for every initial conditions $(u_0,v_0) \in\,\mathcal{H}_1$, there exists a unique adapted process $(u_\mu,v_\mu) \in\,L^2(\Omega,C([0,T];\mathcal{H}_1))$ which solves the systems of equations
\begin{equation}
\label{sm20}
\left\{\begin{array}{l}
\ds{u_\mu(t,x)=u_0(x)+\int_0^t v_\mu(s,x)\,ds,}\\
\vs
\ds{\mu\,v_\mu(t,x)=\mu\,v_0(x)+\int_0^t\left[\Delta u_\mu(s,x)-\gamma(u_\mu(s,x))v_\mu(s,x)+f(x,u_\mu(s,x))+\sigma(u_\mu(s))Q\varphi(s,x)\right]\,ds}\\
\vs
\ds{\quad\quad\quad\quad\quad+\sqrt{\mu}\int_0^t \si(u_\mu(s))\,dw^Q(s).}
\end{array}\right.
	\end{equation}
 	
\end{Theorem}

Once proved Theorem \ref{well-posedness}, 
we introduce the following two conditions.
\begin{enumerate}

\item[C1.] Let $\{\varphi_\mu\}_{\mu>0}$ be an arbitrary family of processe in $ \Lambda_{T,M}$ such that 
\[\lim_{\mu\to 0} \varphi_\mu=\varphi,\ \ \ \ \text{in distribution in}\ \ \ L_w^2(0,T;H),\]
where $L^2_w(0,T;H)$ is the space $L^2(0,T;H)$ endowed with the weak topology and  $\varphi \in\,\Lambda_{T,M}$. Then, for every $p<\infty$ we have
\[\lim_{\mu\to 0} u^{\varphi_\mu}_\mu=u^\varphi,\ \ \ \ \text{weakly in}\ \ \ C([0,T],L^p(\mathcal{O})),\]
where $u^{\varphi_\mu}_\mu$ is the solution to \eqref{SPDE-controlled}, corresponding to the control $\varphi_\mu$,  and $u^\varphi$ is the solution to \eqref{sm18}, corresponding to the control $\varphi$.
\item[C2.] For every $T, R>0$ and $p<\infty$, the level sets $\Phi_{T,R}=\{I_T	\leq R\}$ are compact in the space $C([0,T];L^p(\mathcal{O}))$.

\end{enumerate}

As shown in \cite{bdm}, 
Conditions  C1. and C2. imply the validity of a Laplace principle with action functional $I_T$ in the space $C([0,T];L^p(\mathcal{O}))$ for the family $\{u_\mu\}_{\mu>0}$. Due to the compactness of the level sets $\Phi_{T, R}$ stated in C1. this is equivalent to the validity of Theorem \ref{mainteo}.

Thus, in what follows our strategy will be first proving Theorem \ref{well-posedness}, for every fixed $\mu>0$, and then proving conditions C1. and C2.

\section{Well-posedness of equation \eqref{sm20} }
\label{sec4}

In Theorem \ref{well-posedness} the parameter $\mu>0$ is fixed. This means that in this section, without any loss of generality, we can assume $\mu=1$.
If we denote
\[\eta:=v+g(u),\ \ \ \ z=(u,\eta),\]
then system \eqref{sm20} can be rewitten as the following abstract stochastic evolution equation
\begin{equation}
\label{abstract}	
dz(t)=\left[A(z(t))+B_\varphi(t,z(t))\right]\,dt+\Sigma(z(t))dw^Q(t),\ \ \ \ z(0)=(u_0,v_0+g(u_0)),
\end{equation}
where
\begin{equation*}
	A(u,\eta)=\left(-g(u)+\eta,\Delta u+F(u)\right),\ \ \ \ \ \ \ (u,\eta)\in D(A)=\mathcal{H}_{1},
\end{equation*}

\begin{equation*}
	B_\varphi(t,(u,\eta))=(0,\sigma(u)Q\varphi(t)),\ \ \ \ \ \ \ \ (u,\eta)\in \mathcal{H},\ \ \ \ t\in[0,T],
\end{equation*}

and

\begin{equation*}
	\Sigma(u,\eta)=(0,\sigma(u)),\ \ \ \ \ \ \ \ (u,\eta)\in \mathcal{H}.
\end{equation*}
This means that the adapted $\mathcal{H}_1$-valued process $z(t)=(u(t),\eta(t))$ is the unique solution to the equation
\begin{equation}
\label{sm21}
z(t)=(u_0,v_0+g(u_0))+\int_0^t[A(z(s))+B_\varphi(s,z(s))]\,ds+\int_0^t \Sigma(z(s))\,dw^Q(s),	
\end{equation}
if and only if the adapted $\mathcal{H}_1$-valued process $(u(t),v(t))=(u(t),-g(u(t))+\eta(t))$ is the unique solution of 
\begin{equation}
\label{sm23}
\left\{\begin{array}{l}
\ds{u(t,x)=u_0(x)+\int_0^t v(s,x)\,ds,}\\
\vs
\ds{v(t,x)=\mu\,v_0(x)+\int_0^t\left[\Delta u(s,x)-\gamma(u(s,x))v(s,x)+f(x,u(s,x))+\sigma(u(s,\cdot))Q\varphi(s,x)\right]\,ds}\\
\vs
\ds{\quad\quad\quad+\int_0^t \si(u(s,\cdot))\,dw^Q(s).}
\end{array}\right.
	\end{equation}

In our proof of Theorem \ref{well-posedness} we first assume that $f:\mathbb{R}\to \mathbb{R}$ is Lipschitz-continuous and then we consider the case Hypothesis \ref{Hypothesis3} holds.

\subsection{The case  when $f$ satisfies Hypothesis \ref{Hypothesis3-bis}}
\label{subsection4.1}

In \cite[Section 3]{CX} an analogous result has been proved, in the case $\varphi=0$ (and hence $B_\varphi=0$) and   $\sigma$. Here, we extend the arguments used in \cite{CX} to consider the case of an arbitrary $\varphi \in\,\Lambda_{T,M}$, $\sigma$ having linear growth.  As in \cite[Section 3]{CX}, the arguments we are using here are based on classical tools from the theory of monotone non-linear operators and we refer to the monograph \cite{Barbu} for all details.

Since $f$ is assumed to be Lipschitz continuous, we have  
\begin{equation}
\label{sm30}
\Vert A(z)\Vert_{\mathcal{H}}\leq c\left(1+\Vert z\Vert_{\mathcal{H}_1}\right),\ \ \ \ \ z \in\,D(A),\end{equation}
and, as shown in \cite[Lemma 3.1]{CX}, there exists $\kappa \in\,\mathbb{R}$ such that 
\[\langle A(z_1)-A(z_2),z_1-z_2\rangle_{\mathcal{H}}\leq \kappa\,\Vert z_1-z_2\Vert^2_{\mathcal{H}}.\]
Moreover, for every $\la>0$ small enough 
\[\text{Range}\,(I-\la A)=\mathcal{H}.\]
This means that  the operator $A:D(A)\subset \mathcal{H}\to\mathcal{H}$ is {\em quasi m-dissipative}. In particular, this implies that there exists $\lambda_0>$ such that  
\[J_\la:=\left(I-\la A\right)^{-1},\ \ \ \ \la \in\,(0,\la_0),\]
is a well-defined Lipschitz-continuous mapping in $\mathcal{H}$ and
we can introduce the {\em Yosida approximation} of $A$, defined as
\[A_\lambda(z):=A(J_\la(z))=\frac 1\la \left(J_\la(z)-z\right),\ \ \ \ \ z \in\,\mathcal{H}.\]
Notice that
 \begin{equation}
 \label{diss_Yodida}
 	\langle A_{\lambda}(z_{1})-A_{\lambda}(z_{2}),z_{1}-z_{2}\rangle_{\mathcal{H}}\leq\frac{\kappa}{1-\lambda\kappa}\,\Vert z_{1}-z_{2}\Vert^2_{\mathcal{H}},
 \end{equation}
 and
 \begin{equation}\label{Lip_Yosida}
 	\Vert A_{\lambda}(z_{1})-A_{\lambda}(z_{2})\Vert_{\mathcal{H}}\leq \frac{2}{\lambda(1-\lambda\kappa)}\Vert z_{1}-z_{2}\Vert_{\mathcal{H}}.
 \end{equation}
 Moreover, for every $z\in D(A)$ 
 \[
 	\Vert A^{\lambda}(z)\Vert_{\mathcal{H}}\leq\frac{1}{1-\lambda\kappa}\,\Vert A(z)\Vert _{\mathcal{H}},
 \]
so that for every $z \in\,D(A)$
 \[
 	\Vert J_{\lambda}(z)-z\Vert_{\mathcal{H}}=\lambda\Vert A^{\lambda}(z)\Vert\leq\frac{\lambda}{1-\lambda\kappa}\Vert A(z)\Vert_{\mathcal{H}},
 \]
 and 
 \[
 	\lim\limits_{\lambda\to0}\Vert A_{\lambda}(z)-A(z)\Vert _{\mathcal{H}}=0.
 \]
 In \cite[Proof of Theorem 3.2]{CX}, it has been shown that there exists some $\la_1 \in\,(0,\la_0)$ such that for every $\la \in\,(0,\la_1)$ 
 \begin{equation}
 \label{sm41}
 \langle A_\la(z),z\rangle_{\mathcal{H}_1}\leq -\frac{\gamma_0}2\,\Vert J_\la(z)_1\Vert^2_{H^1}+c\,\Vert J_\la(z)_2\Vert_{H^{-1}}^2,\ \ \ \ \ z \in\,\mathcal{H}_1.	
 \end{equation}
 Furthermore, for every $\la, \nu \in\,(0,\la_0)$  and $z_1, z_2 \in\,\mathcal{H}_1$ it holds
 \begin{equation}
 \label{sm49}
 \begin{array}{l}
 \ds{\langle A_\la(z_1)-A_\nu(z_2),z_1-z_2\rangle_{\mathcal{H}}\leq \Vert z_1-z_2\Vert_{\mathcal{H}}^2+c\left(\la+\nu\right)\left(\Vert z_1\Vert_{\mathcal{H}_1}^2+\Vert z_2\Vert_{\mathcal{H}_1}^2	+1\right).}
 \end{array}
	\end{equation}

Concerning the random  operator $B_\varphi$, according to Hypothesis \ref{Hypothesis1} for every $t \in\,[0,T]$ we have that  $B_\varphi(t,\cdot):\mathcal{H}\to\mathcal{H}_{1}$ is well defined and, in view of \eqref{sm2},  for any $z_{1}, z_{2}\,\in\mathcal{H}$
\begin{equation}
\label{sm25}
	\begin{aligned}
		\Vert B_\varphi(t,z_{1})-B_\varphi(t,z_{2})\Vert _{\mathcal{H}_{1}}
		&=\Vert(\sigma(u_{1})-\sigma(u_{2}))\,Q\varphi(t)\Vert_{H}\leq \sqrt{L}\,\Vert u_1-u_2\Vert_H\,\Vert\varphi(t)\Vert_H.
	\end{aligned}
\end{equation}
Finally, Hypothesis \ref{Hypothesis1} implies that the mapping $\Sigma:\mathcal{H}\to \mathcal{L}_2(H_Q,\mathcal{H}_1)$ is well-defined and due to \eqref{sm2} for any $z_{1}, z_{2}\,\in\mathcal{H}$
\begin{equation}
	\label{sm26}
	\Vert \Sigma(z_1)-\Sigma(z_2)\Vert_{\mathcal{L}_2(H_Q,\mathcal{H}_1)}=\Vert \si(u_1)-\si(u_2)\Vert_{\mathcal{L}_2(H_Q,H)}\leq \sqrt{L}\,\Vert u_1-u_2\Vert_H.
\end{equation}

\bigskip

{\em Step 1.}
For every $\la \in\,(0,\la_0)$  the approximating problem
\begin{equation}
\label{sm27}	dz(t)=\left[A_\la(z(t))+B_\varphi(t,z(t))\right]\,dt+\Sigma(z(t))dw^Q(t),\ \ \ \ z(0)=(u_0,v_0+g(u_0))
\end{equation}
admits a unique solution $z_\la \in\,L^2(\Omega;C([0,T];\mathcal{H}))$.
\smallskip

{\em Proof of Step 1.}
According to  \eqref{Lip_Yosida}, we have 
\begin{equation}
\label{sm32}
\int_{0}^{T}\Vert A_{\lambda}(z_1(s))-A_{\lambda}(z_2(s))\Vert_{\mathcal{H}}^2\,ds
 	\leq \frac{c}{\lambda^2}\int_{0}^{T}\Vert z_1(s)-z_2(s)\Vert_{\mathcal{H}}^2\,ds\leq \frac{c_{T}}{\la^2}\sup_{t\in[0,T]}\Vert z_1(t)-z_2(t)\Vert_{\mathcal{H}}^2.\end{equation}
 	According to \eqref{sm25}, if $\varphi \in\,\Lambda_{T,M}$, we have
 	\begin{equation}
 	\label{sm33}	
 	\begin{array}{ll}	
\ds{\int_0^T\Vert B_\varphi(t,z_{1})-B_\varphi(t,z_{2})\Vert^2 _{\mathcal{H}_{1}}\,dt}  &  \ds{\leq L \int_0^T \Vert u_1(t)-u_2(t)\Vert^2_H\,\Vert\varphi(t)\Vert^2_H\,dt}\\
& \vs
& \ds{\leq L\,M^2\sup_{t\in[0,T]}\Vert z_1(t)-z_2(t)\Vert_{\mathcal{H}}^2,\ \ \ \ \ \mathbb{P}-\text{a.s}.}
\end{array}
\end{equation}
Finally, according  to \eqref{sm26}, we have
\begin{equation}
\label{sm34}
\begin{array}{ll}
\ds{\mathbb{E}\sup_{t \in\,[0,T]}\left \Vert \int_0^t(\Sigma(z_{1}(s))-\Sigma(z_{2}(s))dw^Q(s)\right\Vert^2_{\mathcal{H}_1}}  &  \ds{\leq c\int_0^T\mathbb{E}\,\Vert\Sigma(z_{1}(s))-\Sigma(z_{2}(s))\Vert_{\mathcal{L}_{2}(H_{Q},\mathcal{H}_1)}^2\,ds	}\\
& \vs
&\ds{\leq T L\,\mathbb{E}\sup_{t\in[0,T]}\Vert z_1(t)-z_2(t)\Vert_{\mathcal{H}}^2.}
\end{array}
	\end{equation}
	Therefore, in view of \eqref{sm32}, \eqref{sm33} and \eqref{sm34}, for every $\la \in\,(0,\la_0)$  the mapping
	\[\Phi_\la(z)(t)=(u_0,v_0+g(u_0))+\int_0^t[A_\la(z(s))+B_\varphi(s,z(s))]\,ds+\int_0^t \Sigma(z(s))\,dw^Q(s)\]
	is Lipschitz continuous from $L^2(\Omega;C([0,T];\mathcal{H}))$ into itself, and then for every $\la \in\,(0,\la_0)$ equation \eqref{sm27} admits   a unique solution $z_\la \in\,L^2(\Omega;C([0,T];\mathcal{H}))$.
	
	\bigskip
	
	{\em Step 2.} There exists $c_T>0$ such that		\begin{equation}
	\label{sm35}
	\mathbb{E}\sup_{t \in\,[0,T]}\Vert z_\la(t)\Vert^2_{\mathcal{H}_1}\leq c_T\left(1+\Vert z_0\Vert_{\mathcal{H}_1}^2\right),\ \ \ \ \ \la \in\,(0,\la_1).
	\end{equation}
	
	\smallskip
	
	{\em Proof of Step 2.} As a consequence of It\^o's formula, we have
	\begin{align*}
 	\Vert z_{\lambda}(t)\Vert_{\mathcal{H}_1}^2
 	&=\Vert z_{0}\Vert_{\mathcal{H}_1}^2+2\int_{0}^{t}\langle A_{\lambda}(z_{\lambda}(s)),z_{\lambda}(s)\rangle_{\mathcal{H}_1}\,ds+2\int_{0}^{t}\langle B_\varphi(s,z_{\lambda}(s)),z_{\lambda}(s)\rangle_{\mathcal{H}_1}\,ds\\
 	&\vs
 	&\ \ \ \  +\int_{0}^{t}\Vert\Sigma(z_{\lambda}(s))\Vert_{\mathcal{L}_{2}(H_{Q},\mathcal{H}_1)}^2ds+2\int_{0}^{t}\langle z_{\lambda}(s),\Sigma(z_{\lambda}(s))dw^{Q}(s)\rangle_{\mathcal{H}_1}.
 \end{align*}
	Due to \eqref{sm41}, we have
	\begin{equation}
	\label{sm44}
	\begin{array}{ll}
\ds{\int_{0}^{t}\langle A_{\lambda}(z_{\lambda}(s)),z_{\lambda}(s)\rangle_{H^{-1}}\,ds}
 &  \ds{\leq -\frac{\gamma_0}2\,\int_0^t\Vert J_\la(z_\la(s))_1\Vert^2_{H^1}\,ds+c\int_0^t\Vert J_\la(z_\la(s))\Vert_{\mathcal{H}}^2\,ds}\\
&\vs
& \ds{\leq -\frac{\gamma_0}2\,\int_0^t\Vert J_\la(z_\la(s))_1\Vert^2_{H^1}\,ds+\int_0^t\Vert z_\la(s)\Vert_{\mathcal{H}}^2\,ds.}
\end{array}	
\end{equation}
Next, 
 recalling that $\varphi \in\,\Lambda_{T,M}$, due to \eqref{sm25} for every $\d>0$ we have
 \begin{equation}
 \label{sm39}
 \begin{array}{l}
\ds{\left|\int_{0}^{t}\langle B_\varphi(s,z_{\lambda}(s)),z_{\lambda}(s)\rangle_{\mathcal{H}_1}ds\right|\leq \d\int_0^t\Vert B_\varphi(s,z_\la(s))\Vert^2_{\mathcal{H}_1}\,ds+\frac c\delta  \int_0^t\Vert z_\la(s)\Vert^2_{\mathcal{H}_1}\,ds}\\
\vs\ds{\leq \d\,c_M\left(1+\sup_{r \in\,[0,t]}\Vert z_\la(r)\Vert^2_{\mathcal{H}}\right) + \frac c\delta  \int_0^t\Vert z_\la(s)\Vert^2_{\mathcal{H}_1}\,ds,\ \ \ \ \ \mathbb{P}-\text{a.s}.}
	\end{array}
	\end{equation}
	In the same way, thanks to \eqref{sm26} for every $\d>0$
\begin{equation}
\label{sm40}
\begin{array}{l}
 \ds{\mathbb{E}\sup_{r\in[0,t]}\left\vert \int_{0}^{r}\langle z_{\lambda}(s),\Sigma(z_{\lambda}(s))dw^{Q}(s)\rangle_{\mathcal{H}_1}\right\vert\leq c\,\mathbb{E}\left(\sup_{r\in[0,t]}\Vert z_{\lambda}(r)\Vert_{\mathcal{H}_1}^2\int_{0}^{t}\Vert\Sigma(z_{\lambda}(s))\Vert_{\mathcal{L}_{2}(H_{Q},\mathcal{H}_1)}^2 ds\right)^{\frac{1}{2}}}\\
 \vs
 	\ds{\leq \d\,\mathbb{E}\sup_{r\in[0,t]}\Vert z_{\lambda}(r)\Vert_{\mathcal{H}_1}^2+\frac{c}{\d}\int_{0}^{t}\mathbb{E}\Vert z_{\lambda}(s)\Vert_{\mathcal{H}_1}^2 \,ds+\frac{c_T}\d.}
 \end{array}
 \end{equation}
 Finally, due to \eqref{sm26} we get
 \begin{equation}
 \label{sm45}
 \int_{0}^{t}\Vert\Sigma(z_{\lambda}(s))\Vert_{\mathcal{L}_{2}(H_{Q},\mathcal{H}_1)}^2\,ds	\leq c_T\left(1+\int_0^t \mathbb{E}\Vert z_\la(s)\Vert_{\mathcal{H}}^2\,ds\right)
 \end{equation}

Therefore, if we choose  $\d>0$ small enough, from \eqref{sm44}, \eqref{sm39}, \eqref{sm40} and \eqref{sm45} we get
 \[
 \begin{array}{l}
\ds{	\mathbb{E}\sup_{ r \in\,[0,t]}\Vert z_\la(r)\Vert^2_{\mathcal{H}_1}+\frac{\gamma_0}2\,\int_0^t\mathbb{E}\Vert J_\la(z_\la(s))_1\Vert^2_{H^1}\,ds\leq c_T\int_0^t \mathbb{E}\,\sup_{r \in\,[0,s]}\Vert z_\la(r)\Vert_{\mathcal{H}_1}^2\,ds+c_T.}
\end{array}
\]
Now, Gronwall's lemma allows to obtain \eqref{sm35}.

\bigskip

{\em Step 3.}
There exists $z \in\,L^2(\Omega; C([0,T];\mathcal{H}))$ such that
\begin{equation}
\label{sm46}
\lim_{\la\to 0} \mathbb{E} \sup_{t \in\,[0,T]}\Vert z_\la(t)-z(t)\Vert_{\mathcal{H}}^2=0.
\end{equation}
	
\smallskip
	
	{\em Proof of Step 3.} 	For every $\la, \nu \in\,(0,\la_1)$, we denote $\rho_{\la,\nu}(t):=z_\la(t)-z_\nu(t)$. We have
	 \begin{equation*}
 \begin{array}{ll}
 \ds{d\rho_{\lambda,\nu}(t)=}  &  \ds{\big[A_{\lambda}(z_{\lambda}(t))-A_{\nu}(z_{\nu}(t))\big]dt}\\
 & \vs
 &\ds{	+\big[B_\varphi(t,z_{\lambda}(t))-B_\varphi(t,z_{\nu}(t))\big]dt+\big[\Sigma(z_{\lambda}(t))-\Sigma(z_{\nu}(t))\big]dw^{Q}(t).}
 \end{array}
 \end{equation*}
We have
 \[\begin{array}{l}
 	\ds{\Vert\rho_{\lambda,\nu}(t)\Vert_{\mathcal{H}}^2=2\int_{0}^{t}\langle A_{\lambda}(z_{\lambda}(s))-A_{\nu}(z_{\nu}(s)),\rho_{\lambda,\nu}(s)\rangle_{\mathcal{H}}ds}\\
 	\vs
 	\ds{+2\int_{0}^{t}\langle B_\varphi(s,z_{\lambda}(s))-B_\varphi(s,z_{\nu}(s)),\rho_{\lambda,\nu}(s)\rangle_{\mathcal{H}}ds+\int_{0}^{t}\Vert\Sigma(z_{\lambda}(s))-\Sigma(z_{\nu}(s))\Vert_{\mathcal{L}_{2}(H_{Q},\mathcal{H})}^2\,ds}\\
 	\vs
 	\ds{+2\int_{0}^{t}\langle \rho_{\lambda,\nu}(s),\left(\Sigma(z_{\lambda}(s))-\Sigma(z_{\nu}(s))\right)dw^{Q}(s)\rangle_{\mathcal{H}}=:\sum_{k=1}^4 I_k(t).}
 \end{array}\]
 Due to \eqref{sm49}, we have
 \begin{equation}
 \label{sm50}
 |I_1(t)| \leq c\int_{0}^{t}\Vert\rho_{\lambda,\nu}(s)\Vert_{\mathcal{H}}^2ds+c(\lambda+\nu)\int_{0}^{t}\left(\Vert z_{\lambda}(s)\Vert_{\mathcal{H}_{1}}^2+\Vert z_{\nu}(s)\Vert_{\mathcal{H}_{1}}^2+1\right)ds,	
 \end{equation}
 and due to \eqref{sm26} we have
 \begin{equation}
 \label{sm51}
 |I_3(t)|\leq c\int_0^t \Vert \rho_{\la, \nu}(s)\Vert_{\mathcal{H}}^2\,ds.
 \end{equation}
 Moreover, by proceeding as in the proof of \eqref{sm39} and \eqref{sm40}, for every $\d>0$ we have
 \begin{equation}
 \label{sm52}
 \mathbb{E}\sup_{r \in\,[0,t]}\vert I_2(r)\vert +\mathbb{E}\sup_{r \in\,[0,t]}\vert I_4(r)\vert\leq \d\, \mathbb{E}\sup_{r \in\,[0,t]}\Vert \rho_{\la, \nu}(r)\Vert_{\mathcal{H}}^2+\frac c\d\int_0^t \mathbb{E}\,\Vert \rho_{\la, \nu}(r)\Vert_{\mathcal{H}}^2\,dr.
 \end{equation}
Therefore, if we choose $\d>0$ sufficiently small, from \eqref{sm50}, \eqref{sm51} and \eqref{sm52}, we obtain
\[\begin{array}{ll}
\ds{\mathbb{E}	\sup_{r \in\,[0,t]}\Vert\rho_{\lambda,\nu}(r)\Vert_{\mathcal{H}}^2\leq}  &  \ds{ c\int_0^t \mathbb{E}\,\sup_{r \in\,[0,s]}\Vert \rho_{\la, \nu}(r)\Vert_{\mathcal{H}}^2\,ds}\\
& \vs
&\ds{+c\,(\lambda+\nu)\int_{0}^{t}\left(\Vert z_{\lambda}(s)\Vert_{\mathcal{H}_{1}}^2+\Vert z_{\nu}(s)\Vert_{\mathcal{H}_{1}}^2+1\right)ds,}
\end{array}\]
and the Gronwall lemma, together with \eqref{sm35}, gives 
\[\begin{array}{ll}
\ds{\mathbb{E}	\sup_{r \in\,[0,T]}\Vert\rho_{\lambda,\nu}(r)\Vert_{\mathcal{H}}^2}  &  \ds{	\leq c_T(\lambda+\nu)\int_{0}^{T}\left(\Vert z_{\lambda}(s)\Vert_{\mathcal{H}_{1}}^2+\Vert z_{\nu}(s)\Vert_{\mathcal{H}_{1}}^2+1\right)ds}\\
&\vs
&\ds{\leq c_T(\la+\nu)\left(1+\Vert z_0\Vert_{\mathcal{H}}^2\right).}
\end{array}
\]
This implies that
\[\lim_{\la, \nu\to 0}\mathbb{E}	\sup_{r \in\,[0,T]}\Vert\rho_{\lambda,\nu}(r)\Vert_{\mathcal{H}}^2=0,\]
so that the family $\{z_\la\}_{\la \in\,(0,\la_1)}$ is Cauchy and \eqref{sm46} follows.

\bigskip

{\em Step 4.} There exists a unique solution $z \in\,L^2(\Omega;C([0,T];\mathcal{H}_1))$ for equation \eqref{abstract}.

\smallskip

{\em Proof of step 4.} For every $\la \in\,(0,\la_0)$ we have that $z_\la$ satisfies equation \eqref{sm27}. Then, by proceeding as in \cite[Proof of Theorem 3.2]{CX},  we take the limit, as $\la$ goes to zero, of both sides of \eqref{sm27} in $L^2(\Omega;C([0,T];\mathcal{H}_{-1}))$ and thanks to \eqref{sm46} we obtain that $z$ satisfies the equation
\[z(t)=(u_0,v_0+g(u_0))+\int_0^t[A(z(s))+B_\varphi(s,z(s))]\,ds+\int_0^t \Sigma(z(s))\,dw^Q(s),\]
 and $z \in\,L^2(\Omega;L^\infty(0,T;\mathcal{H}_1))$.

Next, by using again arguments analogous to those used in \cite[Proof of Theorem 3.2]{CX}, we can show  that $z$ has continuous trajectories and is the unique solution of equation \eqref{sm21}.

\subsection{The case when $f$ satisfies Hypothesis \ref{Hypothesis3}}
\label{subsection4.2}

In view of what we have seen in Subsection \ref{subsection4.1}, for every $n \in\,\nat$ and for every $\varphi \in\,\Lambda_{T, M}$ and $(u_0, v_0) \in\,\mathcal{H}_1$ there exists a unique solution $(u_n,v_n) \in\,L^2(\Omega;C([0,T];\mathcal{H}_1))$ for the equation 
\begin{equation}
\label{SPDE-controlled-alpha}
\left\{\begin{array}{l}
\ds{u_n(t,x)=u_0(x)+\int_0^t v_n(s,x)\,ds,}\\
\vs
\ds{v_n(t,x)=\mu\,v_0(x)+\int_0^t\left[\Delta u_n(s,x)-\gamma(u_n(s,x))v_n(s,x)+f_n(x,u(s,x))+\sigma(u_n(s,\cdot))Q\varphi(s,x)\right]\,ds}\\
\vs
\ds{\quad \quad\quad\quad+\int_0^t \si(u(s,\cdot))\,dw^Q(s,x),}
\end{array}\right.
\end{equation}
where $f_n$ is the function defined in \eqref{sm300}.

\bigskip

For every $n\in\mathbb{N}$, we define
\[\tau_n:=\inf\,\left\{t\geq 0\ :\ \Vert u_n(t)\Vert_{H^1}\geq n/C\right\},\]
with $\inf \emptyset=+\infty$, where $C>0$ is a constant such that $\norm{\cdot}_{L^{\infty}(0,L)}\leq C\norm{\cdot}_{H}$. Clearly $\{\tau_n\}_{n \in\,\mathbb{N}}$ is an increasing sequence of stopping times.\\

We denote $\tau:=\sup_{n \in\,\mathbb{N}}\tau_{n}$, and for every $\omega\in\Omega$ and $t<\tau(\omega)\wedge T$, define
\begin{equation*}
	z(t)(\omega):=z_{n}(t)(\omega),\ \ \ \text{if}\ t<\tau_{n}(\omega)\leq T.
\end{equation*}
Notice that this is a good definition, as $f_n(r)=f_m(r)$, for every $n\leq m$ and $|r|\leq n$. Moreover, since for $\omega \in\Omega$ and $t\leq \tau_n(\omega)\wedge T$
\[\Vert u_n(t)(\omega)\Vert_{L^\infty([0,L])}\leq \Vert u_n(t)(\omega)\Vert_{H^1}\leq n,\]
we have that $f_n(u(s)(\omega))=f(u(s)(\omega))$. This means 
$z(t)=z_{n}(t)$ solves equation \eqref{sm23} for $t\leq \tau_{n}\wedge T$.\\

{\em Step 1.} There exists $c_{T}>0$ independent of $n\in\mathbb{N}$ such that
\begin{equation}\label{sm148}
	\begin{array}{l}
		\ds{\E\sup_{t\in[0,T]}\norm{u(t\wedge\tau_{n})}_{H}^{2}+\int_{0}^{T}\E\norm{u(t\wedge\tau_{n})}_{H^{1}}^{2}dt+\int_{0}^{T}\E\norm{u(t\wedge\tau_{n})}_{L^{\theta+1}}^{\theta+1}dt}\\
		\vs
		\ds{\leq c_{T}\Big(1+\int_{0}^{T}\E\norm{u(t\wedge\tau_{n})}_{H}^{2}dt+\int_{0}^{T}\E\norm{v(t\wedge\tau_{n})}_{H}^{2}dt+\E\sup_{t\in[0,T]}\norm{v(t\wedge\tau_{n})}_{H}^{2}\Big).}
	\end{array}
\end{equation}

\medskip

{\em Proof of Step 1.} Recall that for $t< \tau_{n}\leq T$, $z(t)=z_{n}(t)$ is a solution of equation \eqref{sm23}, by proceeding as in \cite[Proof of Lemma 4.1]{CX}, we have 
\begin{equation}
	\label{sm146}
	\begin{array}{l}
		\ds{\frac{\gamma_{0}}{4}\norm{u(t\wedge\tau_{n})}_{H}^{2}+\int_{0}^{t\wedge\tau_{n}}\norm{u(s)}_{H^{1}}^{2}\leq c+c\norm{v(t\wedge\tau_{n})}_{H}^{2}+\int_{0}^{t\wedge\tau_{n}}\norm{v(s)}_{H}^{2}ds}\\
		\vs
		\ds{+\int_{0}^{t\wedge\tau_{n}}\inner{F(u(s)),u(s)}_{H}ds+\int_{0}^{t\wedge\tau_{n}}\inner{u(s),\sigma(u(s))Q\varphi(s)}_{H}ds+\int_{0}^{t\wedge\tau_{n}}\inner{u(s),\sigma(u(s))dw^{Q}(s)}_{H}}.
	\end{array}
\end{equation}
Thanks to \eqref{sm145}, we have 
\begin{equation}
	\label{sm147}
	\int_{0}^{t\wedge\tau_{n}}\inner{F(u(s)),u(s)}_{H}ds\leq -c_{2}\int_{0}^{t\wedge\tau_{n}}\norm{u(s)}_{L^{\theta+1}}^{\theta+1}ds+c_{2}t
\end{equation}
Moreover, by proceeding as in the proof of \eqref{sm39} and \eqref{sm40}, for every $\d>0$ we have
\[
\begin{array}{l}	
	\ds{\mathbb{E}\sup_{r\in[0,t]}\left|\int_{0}^{r\wedge\tau_{n}}\langle u(s),\sigma(u(s))Q\varphi(s)\rangle_{H}ds\right|+\mathbb{E}\sup_{r\in[0,t]}\left|\int_{0}^{r\wedge\tau_{n}}\langle u(s),\sigma(u(s))dw^{Q}(s)\rangle_{H}\right|}\\
	\vs
	\ds{\leq \d 	\mathbb{E}\sup_{r\in[0,t]}\Vert u(r\wedge\tau_{n})\Vert_{H}^2+ \frac c\d \int_0^{t} \mathbb{E}\Vert u(s\wedge\tau_{n})\Vert_{H}^2\,ds.}
\end{array}
\]
Therefore, if we choose $\d>0$ sufficiently small above, this, together with  \eqref{sm146} and \eqref{sm147} allows to conclude that \eqref{sm148} holds true.

\bigskip

{\em Step 2.} There exists $c_T>0$ independent of $n \in\,\mathbb{N}$ such that
\begin{equation}
	\label{sm149}
\begin{array}{l}
\ds{\mathbb{E}\,	\sup_{t \in \,[0,T]}\Vert(u(t\wedge\tau_{n}),v(t\wedge\tau_{n}))\Vert_{\mathcal{H}_1}^2+\E\sup_{t\in[0,T]}\norm{u(t\wedge\tau_{n})}_{L^{\theta+1}}^{\theta+1}+\gamma_{0}\int_0^T\mathbb{E}\Vert v(s)\Vert_H^2\,ds\leq c_T\left(1+\Vert u_0\Vert_{H^1}^{\theta+1}+\Vert v_0\Vert_{H}^2\right).}	
\end{array}
 \end{equation}

\medskip

{\em Proof of Step 2.} From the It\^o formula we have
\[\begin{array}{l}
\ds{\frac{1}{2}\Big[\Vert u(t\wedge\tau_{n})\Vert_{H^1}^2+\Vert v(t\wedge\tau_{n})\Vert^2_{H}\Big]=\frac{1}{2}\Big[\Vert u_0\Vert_{H^1}^2+\Vert v_0\Vert^2_{H}\Big]-\int_0^{t\wedge\tau_{n}} \langle \gamma(u(s))v(s),v(s)\rangle_H\,ds}\\
\vs
\ds{+\int_{\mathcal{O}}\mathfrak{f}(x,u_n(t\wedge\tau_{n},x))\,dx-\int_{\mathcal{O}}\mathfrak{f}(x,u_0(x))\,dx+\int_0^{t\wedge\tau_{n}}\langle \sigma(u(s))Q\varphi(s),v(s)\rangle_H\,ds}\\
\vs
\ds{+\int_0^{t\wedge\tau_{n}}\langle \sigma(u(s))dw^Q(s),v(s)\rangle_H+\frac{1}{2}\int_{0}^{t\wedge\tau_{n}}\norm{\sigma(u(s))}_{\mathcal{L}_{2}(H_{Q},H)}^{2}ds.}	
\end{array}\]
Thanks to \eqref{sm12} we have
\[\int_{\mathcal{O}}\mathfrak{f}(x,u(t\wedge\tau_{n},x))\,dx\leq c-c_{2}\norm{u(t\wedge\tau_{n})}_{L^{\theta+1}}^{\theta+1},\]
and \[\left|\int_{\mathcal{O}}\mathfrak{f}(x,u_0(x))\,dx\right|\leq c\left(1+\int_{\mathcal{O}}|u_0(x)|^{\theta+1}\,dx\right)=c\left(1+\Vert u_0\Vert_{L^{\theta+1}}^{\theta+1}\right)\leq c\left(1+\Vert u_0\Vert_{H^1}^{\theta+1}\right).\]
Therefore, due to \eqref{sm4},
\[\begin{array}{l}
\ds{\sup_{r \in\,[0,t]}\Vert u(r\wedge\tau_{n})\Vert_{H^1}^2+\sup_{r\in[0,t]}\norm{u(r\wedge\tau_{n})}_{L^{\theta+1}}^{\theta+1}+\sup_{r \in\,[0,t]}\Vert v(r\wedge\tau_{n})\Vert^2_{H}+\gamma_{0}\int_0^{t\wedge\tau_{n}}\Vert v(s)\Vert_H^2\,ds}\\
\vs
\ds{\leq c\left(1+\Vert u_0\Vert_{H^1}^{\theta+1}+\Vert v_0\Vert^2_{H}\right)+\sup_{r \in\,[0,t]}\left|\int_0^{r\wedge\tau_{n}}\langle \sigma(u(s))Q\varphi(s),v(s)\rangle_H\,ds\right|}\\
\vs
\ds{+\sup_{r \in\,[0,t]}\left|\int_0^{r\wedge\tau_{n}}\langle \sigma(u(s))dw^Q(s),v(s)\rangle_H\right|+c\int_{0}^{t\wedge\tau_{n}}\norm{u(s)}_{H}^{2}ds.}
\end{array}\]
Since $\varphi \in\,\Lambda_{T,M}$, by proceeding as in \eqref{sm39} and \eqref{sm40}, this implies that for every $\d>0$ 
\[\begin{array}{l}
\ds{\E\sup_{r \in\,[0,t]}\Vert u(r\wedge\tau_{n})\Vert_{H^1}^2+\E\sup_{r\in[0,t]}\norm{u(r\wedge\tau_{n})}_{L^{\theta+1}}^{\theta+1}+\E\sup_{r \in\,[0,t]}\Vert v(r\wedge\tau_{n})\Vert^2_{H}+\gamma_{0}\E\int_0^{t\wedge\tau_{n}}\Vert v(s)\Vert_H^2\,ds}\\
\vs
\ds{\leq c\left(1+\Vert u_0\Vert_{H^1}^{\theta+1}+\Vert v_0\Vert^2_{H}\right)+\d\,\mathbb{E}\sup_{r \in\,[0,t]}\Vert u(r\wedge\tau_{n})\Vert_{H}^2+\frac c\d\int_0^{t}\,\E\Vert v(s\wedge\tau_{n})\Vert_{H}^2\,ds+c\int_{0}^{t}\E\norm{u(s\wedge\tau_{n})}_{H}^{2}ds.}
\end{array}\]
In particular, if we choose $\d$ small enough we have
\begin{equation}
	\label{sm150}
	\begin{array}{l}
		\ds{\E\sup_{r \in\,[0,t]}\Vert u(r\wedge\tau_{n})\Vert_{H^1}^2+\E\sup_{r\in[0,t]}\norm{u(r\wedge\tau_{n})}_{L^{\theta+1}}^{\theta+1}+\E\sup_{r \in\,[0,t]}\Vert v(r\wedge\tau_{n})\Vert^2_{H}}\\
		\vs
		\ds{\leq c\left(1+\Vert u_0\Vert_{H^1}^{\theta+1}+\Vert v_0\Vert^2_{H}\right)+ c\int_0^{t}\,\E\Vert v(s\wedge\tau_{n})\Vert_{H}^2\,ds+c\int_{0}^{t}\E\norm{u(s\wedge\tau_{n})}_{H}^{2}ds,}
	\end{array}
\end{equation}
and if we choose $\delta$ large enough we have
\begin{equation}
	\label{sm151}
	\begin{array}{l}
		\ds{\E\sup_{r\in[0,t]}\norm{u(r\wedge\tau_{n})}_{L^{\theta+1}}^{\theta+1}+\E\sup_{r \in\,[0,t]}\Vert v(r\wedge\tau_{n})\Vert^2_{H}+\gamma_{0}\E\int_0^{t\wedge\tau_{n}}\Vert v(s)\Vert_H^2\,ds}\\
		\vs
		\ds{\leq c\left(1+\Vert u_0\Vert_{H^1}^{\theta+1}+\Vert v_0\Vert^2_{H}\right)+c\int_{0}^{t}\E\norm{u(s\wedge\tau_{n})}_{H}^{2}ds.}
	\end{array}
\end{equation}
Combining \eqref{sm151} with \eqref{sm148} yields that
\begin{equation}
	\label{sm152}
	\E\sup_{r \in\,[0,t]}\Vert u(r\wedge\tau_{n})\Vert_{H}^2\leq c\int_{0}^{t}\E\norm{u(s\wedge\tau_{n})}_{H}^{2}ds.
\end{equation}
Thanks to Gronwall's lemma, \eqref{sm149} follows from \eqref{sm150}.

\bigskip

{\em Step 3.} There exists $z=(u,v) \in\,L^2(\Omega;C([0,T];\mathcal{H}_1))$ solution to problem \eqref{sm23} such that
\begin{equation}
\label{sm91}
\begin{array}{l}
\ds{\mathbb{E}\,	\sup_{t \in \,[0,T]}\Vert(u(t),v(t))\Vert_{\mathcal{H}_1}^2+\E\sup_{t\in[0,T]}\norm{u(t)}_{L^{\theta+1}}^{\theta+1}+\gamma_{0}\int_0^T\mathbb{E}\Vert v(s)\Vert_H^2\,ds\leq c_T\left(1+\Vert u_0\Vert_{H^1}^{\theta+1}+\Vert v_0\Vert_{H}^2\right).}	
\end{array}
 \end{equation}.

\medskip

{\em Proof of Step 3.} According  to \eqref{sm149}, for every $T>0$ we have
\[\mathbb{P}(\tau_n\leq T)\leq \frac{C^{2}}{n^{2}}\E\big(\norm{u(\tau_{n})}_{H^{1}}^{2};\tau_{n}\leq T\big)\leq \frac {C^{2}}{n^2}\mathbb{E}\sup_{t \in\,[0,T]}\Vert u(t\wedge\tau_{n})\Vert_{H^1}^2\leq \frac c{n^2},\]
so that 
\[\lim_{n\to\infty}\mathbb{P}(\tau_n\leq T)=0,\]	
and hence $\mathbb{P}(\tau=\infty)=1$. This implies for every $t\in[0,T]$, $z(t\wedge\tau_{n})\to z(t)$, $\mathbb{P}$-a.s. as $n\to\infty$, so that $z$ belongs to $L^2(\Omega;C([0,T];\mathcal{H}_1))$ and solves equation \eqref{sm23}. By taking the limit as $n\to\infty$ in \eqref{sm149}, we get \eqref{sm91}.

\bigskip

{\em Step 4.} The solution $z$ is unique in $L^2(\Omega;C([0,T];\mathcal{H}_1))$.

\medskip
{\em Proof of Step 4.} Let $z_1$ and $z_2$ be two solutions of equation \eqref{abstract} in $L^2(\Omega;C([0,T];\mathcal{H}_1))$. For every $R>0$, we define
\[\tau_R:=\tau_{1,R}\wedge \tau_{2,R},\] 
where
\[\tau_{i,R}:=\inf\,\left\{t\geq 0\ :\ \Vert u_i(t)\Vert_{H^1}\geq R\right\},\ \ \ \ i=1, 2.\]
Since $z_1$ and $z_2$ belong to $L^2(\Omega;C([0,T];\mathcal{H}_1))$, we have 
\begin{equation}
\label{sm87}
\lim_{R\to\infty} \mathbb{P}\left(\tau_R<T\right)=0.	
\end{equation}

Now, if we define $\rho=z_1-z_2$, from the It\^o formula we have
\[\begin{array}{l}
 	\ds{\Vert\rho(t\wedge \tau_R)\Vert_{\mathcal{H}}^2=2\int_{0}^{t\wedge \tau_R}\langle A(z_1(s))-A(z_2(s)),\rho(s)\rangle_{\mathcal{H}}ds}\\
 	\vs
 	\ds{+2\int_{0}^{t\wedge \tau_R}\langle B_\varphi(s,z_1(s))-B_\varphi(s,z_2(s)),\rho(s)\rangle_{\mathcal{H}}ds+\int_{0}^{t\wedge \tau_R}\Vert\Sigma(z_1(s))-\Sigma(z_2(s))\Vert_{\mathcal{L}_{2}(H_{Q},\mathcal{H})}^2\,ds}\\
 	\vs
 	\ds{+2\int_{0}^{t\wedge \tau_R}\langle \rho(s),\left(\Sigma(z_1(s))-\Sigma(z_2(s))\right)dw^{Q}(s)\rangle_{\mathcal{H}}=:\sum_{k=1}^4 I_{k}(t\wedge \tau_R).}
 \end{array}\]
We have
\[\begin{array}{ll}
\ds{\langle A(z_1(s))-A(z_2(s)),\rho(s)\rangle_{\mathcal{H}}=}  &  \ds{-\langle(g(u_1(s))-g(u_2(s)),u_1(s)-u_2(s)\rangle_H}\\
& \vs
&\ds{+\langle F(u_1(s))-F(u_2(s)),\eta_1(s)-\eta_2(s)\rangle_{H^{-1}},}
\end{array}\]	
so that, according to \eqref{sm15}
\[\begin{array}{l}
\ds{\left|\langle A(z_1(s))-A(z_2(s)),\rho(s)\rangle_{\mathcal{H}}\right|\leq c\,\Vert \rho(s)\Vert_{\mathcal{H}}^2+c\,\Vert F(u_1(s))-F(u_2(s))\Vert_{H}^2}\\
\vs
\ds{\leq c\,\Vert \rho(s)\Vert_{\mathcal{H}}^2+c\left(1+\Vert u_1(s)\Vert_{H^1}^{2(\theta-1)}+\Vert u_2(s)\Vert_{H^1}^{2(\theta-1)}\right)\Vert u_1(s)-u_2(s)\Vert^2_{H}.}
\end{array}\]
This implies that 
\begin{equation}
\label{sm88}
\sup_{r \in\,[0,t]}|I_{1}(r\wedge \tau_R)|\leq c(R)	\int_0^{t\wedge \tau_R}\Vert \rho(s\wedge \tau_R)\Vert_{\mathcal{H}}^2\,ds.
\end{equation}
Moreover, by proceeding as in \eqref{sm39} and \eqref{sm40}, for every $\d>0$ we have
\begin{equation}
\label{sm89}
\mathbb{E}\sup_{r \in\,[0,t]}\left(|I_2(r\wedge \tau_R)|+|I_4(r\wedge \tau_R)|\right)\leq \d \mathbb{E}\sup_{r \in [0,t]}\Vert \rho(r\wedge \tau_R)\Vert_{\mathcal{H}}^2+\frac c\d \int_0^{t}\mathbb{E}\Vert \rho(s\wedge \tau_R)\Vert_{\mathcal{H}}^2\,ds.	
\end{equation}
Therefore, since 
\[\sup_{r \in\,[0,t]}|I_3(r\wedge \tau_R)|\leq c\int_0^{t\wedge \tau_R}\Vert \rho(s\wedge \tau_R)\Vert_{\mathcal{H}}^2\,ds,\]	
thanks to \eqref{sm88} and \eqref{sm89}, if we fix some $\d>0$ small enough, we get
\[\mathbb{E}\sup_{r \in\,[0,t]}\Vert\rho(r\wedge \tau_R)\Vert_{\mathcal{H}}^2\leq c(R)\int_0^{t}\E\Vert \rho(r\wedge \tau_R)\Vert_{\mathcal{H}}^2\,dr.\]
This implies that for every $R>0$
\[\mathbb{E}\sup_{r \in\,[0,T]}\Vert\rho(r\wedge \tau_R)\Vert_{\mathcal{H}}^2=0.\]
In view of \eqref{sm87}, by taking the limit as $R\uparrow\infty$ this gives
$\mathbb{E}\sup_{r \in\,[0,T]}\Vert\rho(r)\Vert_{\mathcal{H}}^2=0$ and  uniqueness follows.

\section{A-priori bounds and tightness}\label{sec5}
In the previous section we have proved that for any $\mu>0$ and any $T>0$ there exists a unique solution $(u_\mu, \partial_t u_\mu) \in\,L^2(\Omega;C([0,T];\mathcal{H}_1))$ to system \eqref{sm23}. Our purpose here is proving a  bound for $(u_\mu, \partial_t u_\mu)$, which is uniform with respect to $\mu$.

\begin{Lemma}
\label{lemma5.1}
Under Hypotheses \ref{Hypothesis1} and \ref{Hypothesis2} and either Hypothesis \ref{Hypothesis3-bis} or Hypothesis \ref{Hypothesis3}, for every $T, M>0$  and for every initial condition $(u_0,v_0) \in\,\mathcal{H}_1$,  there exist $c_T>0$ and $\mu_T>0$	 such that for every $\varphi \in\,\Lambda_{T,M}$ and $\mu \in\,(0,\mu_T)$
\begin{equation}
\label{sm92}
\mathbb{E}\sup_{t \in\,[0,T]}\Vert u_\mu(t)\Vert_{H^1}^2+\mathbb{E}\sup_{t \in\,[0,T]}\Vert u_\mu(t)\Vert_{L^{\theta+1}}^{\theta+1}+\mu\,\mathbb{E} \sup_{t \in\,[0,T]}\Vert \partial_t u_\mu(t)\Vert_{H}^2+\int_0^T\mathbb{E}\Vert \partial_t u_\mu(t)\Vert_{H}^2\,dt\leq c_T.
\end{equation}
\end{Lemma}

\begin{proof}
We give our proof in case Hypothesis \ref{Hypothesis3} holds and we leave to the reader the proof in case Hypothesis \ref{Hypothesis3-bis} holds. 

\bigskip

{\em Step 1.} There exists $c_T>0$ such that for all $\mu \in\,(0,1)$
\begin{equation}\label{energy1}
		\begin{array}{l}
			\ds{\mathbb{E}\sup_{r\in[0,t]}\Vert u_{\mu}(r)\Vert_{H}^2
			+\int_{0}^{t}\mathbb{E}\Vert u_{\mu}(s)\Vert_{H^{1}}^2\,ds}+\int_{0}^{t}\E\norm{u_{\mu}(s)}_{L^{\theta+1}}^{\theta+1}ds\\
			\vs
			\ds{\leq c_{T}\left(1+\int_{0}^{t}\mathbb{E}\Vert u_{\mu}(s)\Vert_{H}^2ds+\mu\int_{0}^{t}\mathbb{E}\Vert\partial_{t}u_{\mu}(s)\Vert_{H}^2ds+\mu^2\mathbb{E}\sup_{r\in[0,t]}\Vert\partial_{t}u_{\mu}(r)\Vert_{H}^2\right).}
		\end{array}
	\end{equation} 

\medskip
{\em Proof of Step 1.}
As shown in \cite[Proof of Lemma 4.1]{CX}, for every $\mu \in\,(0,1)$ we have
\begin{equation}
\label{sm301}\begin{array}{l}
		\ds{\frac{\gamma_{0}}{4}\Vert u_{\mu}(t)\Vert _{H}^2 +\int_{0}^{t}\Vert u_{\mu}(s)\Vert _{H^{1}}^2ds\leq}
		\ds{ \ c+c\mu^2\Vert \partial_{t}u_{\mu}(t)\Vert _{H}^2+\mu\int_{0}^{t}\Vert \partial_{t}u_{\mu}(s)\Vert _{H}^2ds}\\
		\vs
	\ds{+\int_{0}^{t}\langle F(u_{\mu}(s)),u_{\mu}(s)\rangle_{H}ds+\int_{0}^{t}\langle u_{\mu}(s),\sigma(u_{\mu}(s))Q\varphi_{\mu}(s)\rangle_{H}ds+\sqrt{\mu}\int_{0}^{t}\langle u_{\mu}(s),\sigma(u_{\mu}(s))dw^{Q}(s)\rangle_{H}.}
	\end{array}	
\end{equation}

Due to \eqref{sm145}, we have
\begin{equation}
\label{102}
\int_{0}^{t}\langle F(u_{\mu}(s)),u_{\mu}(s)\rangle_{H}ds\leq -c_{2}\int_{0}^{t}\norm{u_{\mu}(s)}_{L^{\theta+1}}^{\theta+1}ds+c_{2}t.
\end{equation}
	Moreover, by proceeding as in the proof of \eqref{sm39} and \eqref{sm40}, for every $\d>0$ we have
\[
\begin{array}{l}	
\ds{\mathbb{E}\sup_{r\in[0,t]}\left|\int_{0}^{r}\langle u_{\mu}(s),\sigma(u_{\mu}(s))Q\varphi_{\mu}(s)\rangle_{H}ds\right|+\sqrt{\mu}\,\mathbb{E}\sup_{r\in[0,t]}\left|\int_{0}^{t}\langle u_{\mu}(s),\sigma(u_{\mu}(s))dw^{Q}(s)\rangle_{H}\right|}\\
\vs
\ds{\leq \d 	\mathbb{E}\sup_{r\in[0,t]}\Vert u_{\mu}(r)\Vert_{H}^2+ \frac c\d \int_0^t \mathbb{E}\Vert u_{\mu}(s)\Vert_{H}^2\,ds.}
\end{array}
\]
Therefore, if we choose $\d>0$ sufficiently small above, this, together with  \eqref{102} and \eqref{sm301} allows to conclude that \eqref{energy1} holds true.

\bigskip
{\em Step 2.} For every $T>0$ there exist $c_T>0$ and $\mu_T>0$ such that
\begin{equation}
\label{sm101}
\sup_{\mu \in\,(0,\mu_T)} \mathbb{E}\sup_{t \in\,[0,T]}\Vert u_\mu(t)\Vert_{H}^2\leq c_T,	
\end{equation}
and \eqref{sm92} holds.

\medskip

{\em Proof of Step 2.}
As in the previous section, from It\^o's formula, we have
\[\begin{array}{l}
\ds{\Vert u_\mu(t)\Vert_{H^1}^2+\mu \Vert \partial_t u_\mu(t)\Vert^2_{H}=\Vert u_0\Vert_{H^1}^2+\mu \Vert v_0\Vert^2_{H}-\int_0^t \langle \gamma(u_\mu(s))\partial_t u_\mu(s),\partial_t u_\mu(s)\rangle_H\,ds}\\
\vs
\ds{+\int_{\mathcal{O}}\mathfrak{f}(x,u_\mu(t,x))\,dx-\int_{\mathcal{O}}\mathfrak{f}(x,u_0(x))\,dx+\int_0^t\langle \sigma(u_\mu(s))Q\varphi(s),\partial_t u_\mu(s)\rangle_H\,ds}\\
\vs
\ds{+\sqrt{\mu}\int_0^t\langle \sigma(u_\mu(s))dw^Q(s),\partial_t u_\mu(s)\rangle_H+\int_0^t\Vert \sigma(u_\mu(s))\Vert^2_{\mathcal{L}_2(H_Q,H)}\,ds.}	
\end{array}\]
Due to \eqref{sm4}, \eqref{nonlinearity assumption} and \eqref{sm12}, this gives
\[\begin{array}{l}
\ds{\Vert u_\mu(t)\Vert_{H^1}^2+c \Vert u_\mu(s)\Vert_{L^{\theta+1}}^{\theta+1}+\mu \Vert \partial_t u_\mu(t)\Vert^2_{H}+\gamma_0\int_0^t \Vert \partial_t u_\mu(s)\Vert^2_H\,ds}\\
\vs
\ds{\leq c\,\left(1+\Vert u_0\Vert_{H^1}^{\theta+1}+\mu \Vert v_0\Vert^2_{H}\right)+\int_0^t\Vert \si(u_\mu(s))\Vert^2_{\mathcal{L}_2(H_Q,H)}\,ds}\\
\vs
\ds{+\int_0^t\langle \sigma(u_\mu(s))Q\varphi(s),\partial_t u_\mu(s)\rangle_H\,ds+\sqrt{\mu}\int_0^t\langle \sigma(u_\mu(s))dw^Q(s),\partial_t u_\mu(s)\rangle_H.}	
\end{array}\]
Therefore, by proceeding as in Step 1, for every $\d>0$ we have
\begin{equation}
\label{sm104}
\begin{array}{l}
\ds{\mathbb{E}\sup_{r \in\,[0,t]}\Vert u_\mu(r)\Vert_{H^1}^2+\mathbb{E}\sup_{r \in\,[0,T]}\Vert u_\mu(r)\Vert_{L^{\theta+1}}^{\theta+1}+\mu\,\mathbb{E} \sup_{r \in\,[0,t]}\Vert \partial_t u_\mu(r)\Vert_{H}^2+\int_0^t\mathbb{E}\Vert \partial_t u_\mu(s)\Vert_{H}^2\,ds}	\\
\vs\ds{\leq c\,\left(1+\Vert u_0\Vert_{H^1}^{\theta+1}+\mu \Vert v_0\Vert^2_{H}\right)+\d \mathbb{E}\sup_{r \in\,[0,t]}\Vert u_\mu(r)\Vert_{H}^2+\frac c\d \int_0^t\mathbb{E}\Vert \partial_t u_\mu(s)\Vert_{H}^2\,ds+c\int_0^t\E\Vert u_{\mu}(s)\Vert_H^2\,ds.}
\end{array}	
\end{equation}
If we take $\d>0$ in \eqref{sm104},
we get
\begin{equation}
\label{sm105}
\begin{array}{l}
\ds{\mathbb{E}\sup_{r \in\,[0,t]}\Vert u_\mu(r)\Vert_{H^1}^2+\mathbb{E}\sup_{r \in\,[0,T]}\Vert u_\mu(r)\Vert_{L^{\theta+1}}^{\theta+1}+\mu\,\mathbb{E} \sup_{r \in\,[0,t]}\Vert \partial_t u_\mu(r)\Vert_{H}^2	}\\
\vs
\ds{\leq c\left(1+\int_0^t\mathbb{E}\Vert \partial_t u_\mu(s)\Vert_{H}^2\,ds\right)+c\int_0^t\E\Vert u_{\mu}(s)\Vert_H^2\,ds,}	
\end{array}
\end{equation}
and if we take $\d>0$ large enough we get
\begin{equation}
\label{sm106}
\mathbb{E}\sup_{r \in\,[0,T]}\Vert u_\mu(r)\Vert_{L^{\theta+1}}^{\theta+1}+\mu\,\mathbb{E} \sup_{r \in\,[0,t]}\Vert \partial_t u_\mu(r)\Vert_{H}^2+\int_0^t\mathbb{E}\Vert \partial_t u_\mu(s)\Vert_{H}^2\,ds\leq c\left(1+\mathbb{E}\sup_{r \in\,[0,t]}\Vert u_\mu(r)\Vert_{H}^2\right).	
\end{equation}
By combining together \eqref{energy1} and \eqref{sm106}, we can fix $\mu_T>0$ such that for every $\mu \in\,(0,\mu_T)$
\[\mathbb{E}\sup_{r\in[0,t]}\Vert u_{\mu}(r)\Vert_{H}^2\leq c_{T}\left(1+\int_{0}^{t}\mathbb{E}\Vert u_{\mu}(s)\Vert_{H}^2ds\right),\]
which implies \eqref{sm101}. Thus, from \eqref{sm101}, \eqref{sm105} and \eqref{sm106}, we obtain \eqref{sm92}.

\end{proof}

Now, for every $T>0$ and $\mu>0$ we define
\begin{equation*}
	\rho_{\mu}(t,x)=g(u_{\mu}(t,x)),\ \ \ (t,x)\in[0,T]\times \mathcal{O}.
\end{equation*}
According to Hypothesis \ref{Hypothesis2}, we know that 
\begin{equation*}
	|g(r)|\leq \gamma_{1}|r|,\ \ \ \ \ \ \ |g^\prime(r)|\leq \gamma_{1},\ \ \ \ \ \ r \in\,\mathbb{R}
\end{equation*}
so that for every $\mu>0$ and $t\in[0,T]$,
\begin{equation}\label{rho}
	\Vert \rho_{\mu}(t)\Vert_{H}\leq \gamma_{1}\Vert u_{\mu}(t)\Vert_{H},\ \ \ \ \ \ \ \ \ \  \Vert\rho_{\mu}(t)\Vert_{H^{1}}\leq \gamma_{1}\Vert u_{\mu}(t)\Vert_{H^{1}},\ \ \ \ \ \ \ \ \ \  \Vert \partial_t \rho_\mu(t)\Vert_{H}\leq \gamma_1\,\Vert\partial_tu_\mu(t)\Vert_H.
\end{equation}

Since the function $g$ is strictly increasing, it is invertible and  we have 
\begin{equation*}
	u_{\mu}(t,x)=g^{-1}(\rho_{\mu}(t,x)),\ \ \ (t,x)\in[0,T]\times \mathcal{O},
\end{equation*}

which implies that 
\begin{equation*}
	\Delta u_{\mu}(t,x)=\text{div}\left[\nabla g^{-1}(\rho_{\mu}(t))\right]=\text{div}\left[\frac{1}{\gamma(g^{-1}(\rho_{\mu}(t)))}\nabla\rho_{\mu}(t)\right].
\end{equation*}

Moreover, by the definition of $\rho_{\mu}$, 
\begin{equation*}
	\nabla \rho_{\mu}(t)=\gamma(u_{\mu}(t))\nabla u_{\mu}(t),\ \ \ \ \ \ \ \partial_{t}\rho_{\mu}(t)=\gamma(u_{\mu}(t))\partial_{t}u_{\mu}(t).
\end{equation*}

This means that if we integrate  equation (\ref{SPDE-controlled}) (with $\varphi_\mu$) with respect to $t$ we have
\begin{equation}\label{sto4}
	\begin{aligned}
		\rho_{\mu}(t)+\mu\partial_{t}u_{\mu}(t)
		& = g(u_{0})+\mu v_{0}+\int_{0}^{t}\text{div}\big[b(\rho_{\mu}(s))\nabla\rho_{\mu}(s)\big]ds+\int_{0}^{t}F_{g}(\rho_{\mu}(s))ds\\
		&\ \ \ \ +\int_{0}^{t}\sigma_{g}(\rho_{\mu}(s))Q\,\varphi_{\mu}(s)ds+\sqrt{\mu}\int_{0}^{t}\sigma_{g}(\rho_{\mu}(s))dw^{Q}(s),
	\end{aligned}
\end{equation}
where for every $r \in\,\mathbb{R}$ and $x \in\,\mathcal{O}$
\begin{equation} \label{end1}
	b(r):=\frac{1}{\gamma(g^{-1}(r))},\ \ \ \ \ \ \ \ f_{g}(x,r):=f(x,g^{-1}(r)),\end{equation}
and for every   $u \in\,H^1$
\begin{equation}  \label{end2}F_g(u)=f_g\circ u,\ \ \ \ \ \ \ \ \ \ \sigma_{g}(u):=\sigma(g^{-1}\circ u).\end{equation}
\begin{Theorem}
	\label{tightness}
Assume Hypotheses \ref{Hypothesis1} and \ref{Hypothesis2} and either Hypothesis \ref{Hypothesis3-bis} or Hypothesis \ref{Hypothesis3} and fix an arbitrary $T>0$ and $(u_0,v_0) \in\,\mathcal{H}_1$. Then, for any family of predictable controls $\{\varphi_\mu\}_{\mu \in\,(0,\mu_T)}\subset \Lambda_{T, M}$, the family of probabilities $\{\mathcal{L}(\rho_\mu)\}_{\mu \in\,(0,\mu_T)}$ is tight in $C([0,T];H^{\delta})$, for every $\d<1$.
\end{Theorem}

\begin{proof}
According to \eqref{sm92} and \eqref{rho}, we have
that	
\[\mathbb{E}\sup_{r\in[0,t]}\Vert \rho_{\mu}(r)\Vert_{H}^2
			+\int_{0}^{t}\mathbb{E}\Vert \partial_t\rho_{\mu}(s)\Vert_{H^{1}}^2\,ds\leq c_T,\ \ \ \ \ \mu \in\,(0,\mu_T).\]
This means that for every $\e>0$ there exists $L_\e>0$ such that if we denote by $K_\e$ the ball of radius $L_\e$ in $C([0,T];H^1)\cap W^{1,2}([0,T];H)$, then
\[\inf_{\mu \in\,(0,\mu_T)}\mathbb{P}(\rho_\mu \in\,K_\e)\geq 1-\e.\]
This allows to conclude as, due to the Aubin-Lions lemma, the set $K_\e$ is compact in $C([0,T];H^\d)$, for every $\d<1$.
\end{proof}

\section{The limit controlled problem}\label{sec6}

In order to prove conditions C1 and C2,  we need first to understand better the  controlled quasi-linear parabolic problem
\begin{equation}
\label{limit-det}
\left\{\begin{array}{l}
\ds{\partial_t \rho(t,x)=\text{div}\,[b(\rho(t,x))]+f_g(x,\rho(t,x))+\sigma_g(\rho(t,\cdot))\varphi(t,x),\ \ \ \ t>0,	\ \ \ \ \ x \in\,\mathcal{O},}\\
\vs\ds{\rho(0,x)=g(u_0(x)),\ \ \ \ \ \ \ \ \ \rho(t,x)=0,\ \ \ x \in\,\partial \mathcal{O}.}
\end{array}\right.
	\end{equation}

In view of Hypothesis \ref{Hypothesis2}, 
we have
\[\frac 1{\gamma_1}\,|r|\leq |g^{-1}(r)|\leq \frac{1}{\gamma_0},\ \ \ \ r \in\,\mathbb{R}.\]
When Hypothesis \ref{Hypothesis3} holds, thanks to \eqref{sm9}, this means that for every $u \in\,L^{\theta+1}([0,L])$ 
\begin{equation}
\label{sm120}
\Vert F_g(u)\Vert_{H^{-1}}\leq c\int_0^L(1+|g^{-1}(u(x))|^\theta)\,dx\leq c\left(1+\Vert u\Vert_H^{\frac 2{\theta-1}}\Vert u\Vert_{L^{\theta+1}}^{\frac{(\theta+1)(\theta-2)}{\theta-1}}\right).	
\end{equation}
Next, again due to Hypothesis \ref{Hypothesis2}
we have
\begin{equation*}
	\frac 1\gamma_{1}\leq \frac{g^{-1}(r)}{r}\leq \frac 1\gamma_{0},\ \ \ \ \ \ \  \ \ \ \ r \in\,\mathbb{R},
\end{equation*}
so that, thanks to  Hypothesis \ref{Hypothesis3}, we get,
\[
	f_{g}(r)r=f(g^{-1}(r))g^{-1}(r)\cdot\frac{r}{g^{-1}(r)}\leq c\left(1-|r|^{\theta+1}\right),\ \ \ \ \ \ r \in\,\mathbb{R},
\]
(here we define $g^{-1}(0)/0=1/\gamma(0)$). In particular, we have
\begin{equation}
\label{sm108}
\langle F_g(u),u\rangle_H\leq c\left(1-\Vert u\Vert^{\theta+1}_{L^{\theta+1}}\right),\ \ \ \ \ u \in\,H^1.	
\end{equation}
Moreover, since  $g^{-1}$ is  increasing and $f$ is decreasing, we have that $f_g$ is decreasing, so that 
\begin{equation}
\label{sm109}
\langle F_g(u_1)-F_g(u_2),u_1-u_2\rangle_H\leq 0.\end{equation}
Finally, thanks  to Hypotheses \ref{Hypothesis2} and \ref{Hypothesis1},
\begin{equation}
\label{sm112}
|b(r)|\leq \frac 1{\gamma_0},\ \ \ \ \ b(r)\geq \frac 1{\gamma_1},\ \ \ \ \ |b(r)-b(s)|\leq c\,|r-s|,\ \ \ \ \ r,s \in\,\mathbb{R},	\end{equation}
and
\begin{equation}
\label{sm113}
\Vert \si_g(u_1)-\si_g(u_2)\Vert_{\mathcal{L}(H_Q,H)}\leq c\,\Vert u_1-u_2\Vert_H,\ \ \ \ \ \ u_1, u_2 \in\,H.
	\end{equation}

\begin{Definition}
	A function $\rho \in\, L^2(0,T;H^1)$ is a weak solution to equation \eqref{limit-det} if for every test function $\psi \in\,C^\infty_0(\mathcal{O})$ and $t \in\,[0,T]$
	\begin{equation}
	\label{sm123}
		\begin{array}{ll}
			\ds{\langle\rho(t),\psi\rangle_{H}=} & \ds{\langle g(u_{0}),\psi\rangle_H-\int_{0}^{t}\langle b(\rho(s))\nabla\rho(s),\nabla\psi\rangle_{H}ds+\int_{0}^{t}\langle F_{g}(\rho(s))+\sigma_{g}(\rho(s))Q\varphi(s),\psi\rangle_{H}\,ds.}
		\end{array}
	\end{equation}
\end{Definition}

\begin{Theorem}
\label{teo-det}
Assume Hypotheses \ref{Hypothesis1} and \ref{Hypothesis2} and either Hypothesis \ref{Hypothesis3-bis} or Hypothesis \ref{Hypothesis3} and fix any $T>0$ and $	\varphi \in\,L^2(0,T;H)$. Then, for every $u_0 \in\,H^1$ there exists a unique weak solution $\rho$ to equation \eqref{limit-det} such that
\begin{equation}
\label{sm110}
\rho \in\,C([0,T];H)\cap L^2(0,T;H^1),\ \ \ \ \partial_t\rho \in\,L^2(0,T;H^{-1}).	
\end{equation}
\end{Theorem}
\begin{proof}
We prove the result above in case Hypothesis \ref{Hypothesis3} is satisfied. We leave to the reader the proof in the case Hypothesis \ref{Hypothesis3-bis} is satisfied. For every $\varphi \in\,L^2(0,T;H)$ and $\psi \in\,H^1$, equation \eqref{sm123} can be written as
\[\langle\rho(t),\psi\rangle_{H}=\langle g(u_{0}),\psi\rangle_H+\int_0^t \langle S_\varphi(s,\rho(s)),\psi\rangle_H\,ds,\]
where we have defined
\[S_\varphi(t,\rho):=\text{div}\big[b(\rho)\nabla\rho\big]+F_{g}(\rho)+\sigma_{g}(\rho)Q\varphi(t),\ \ t\in[0,T],\ \ \ \ \ \rho\in H^{1}.
	\]
	Due to the presence of $\varphi$, the operator $S_\varphi$ does not satisfy the assumptions  of \cite[Theorem 1.1]{liu} and for this reason we cannot apply directly the well-posedness result proved therein. However, as we are going to show below, the operator $S_\varphi$ satisfies some generalized conditions that still allow to use techniques introduced in \cite{liu} to prove our result.
	
	First of all, we notice that for every $t \in\,[0,T]$ and $\rho_1, \rho_2, \rho \in\,H^1$ the mapping
	$s \in\,\mathbb{R} \mapsto \langle S_\varphi(t,\rho_1+s \rho_2),\rho\rangle_H,$
	is continuous, so that the operator $S_\varphi$ is {\em hemi-continuous}.
	
Thanks to \eqref{sm109}, \eqref{sm112} and \eqref{sm113},   for every  $t\in[0,T]$ and $\rho_{1},\rho_{2}\in\, H^{1}$
	\[\begin{array}{l}
	\ds{\langle\mathcal{S}_\varphi(t,\rho_{1})-\mathcal{S}_\varphi(t,\rho_{2}),\rho_{1}-\rho_{2}\rangle_{H}= \langle\text{div}\left[b(\rho_{1})\nabla\rho_{1}-b(\rho_{2})\nabla\rho_{2}\right],\rho_{1}-\rho_{2}\rangle_{H}}\\
		\vs
		\ds{+\langle F_{g}(\rho_{1})-F_{g}(\rho_{2}),\rho_{1}-\rho_{2}\rangle_{H} +\langle(\sigma_{g}(\rho_{1})-\sigma_{g}(\rho_{2}))Q\varphi(t),\rho_{1}-\rho_{2}\rangle_{H}}\\
		\vs
		\ds{= -\langle b(\rho_{1})\nabla(\rho_{1}-\rho_{2}),\nabla(\rho_{1}-\rho_{2})\rangle_{H} -\langle (b(\rho_{1})-b(\rho_{2}))\nabla\rho_{2},\nabla(\rho_{1}-\rho_{2})\rangle_{H}}\\
		\vs
		\ds{+\langle F_{g}(\rho_{1})-F_{g}(\rho_{2}),\rho_{1}-\rho_{2}\rangle_{H} +\langle(\sigma_{g}(\rho_{1})-\sigma_{g}(\rho_{2}))Q\varphi(t),\rho_{1}-\rho_{2}\rangle_{H}}\\
		\vs
		\ds{\leq -\frac{1}{\gamma_{1}}\Vert\rho_{1}-\rho_{2}\Vert_{H^{1}}^2+c\, \Vert\rho_{1}-\rho_{2}\Vert_{H}^2\,\Vert\rho_{2}\Vert_{H^{1}}^2 +\frac{1}{2\gamma_{1}}\Vert\rho_{1}-\rho_{2}\Vert_{H^{1}}^2+c\left(1+\Vert\varphi(t)\Vert_{H}^2\right)\Vert\rho_{1}-\rho_{2}\Vert_{H}^2}		\end{array}\]
		 so that
		\begin{equation}
		\label{sm114}
		\begin{array}{l}
		\ds{	\langle\mathcal{S}_\varphi(t,\rho_{1})-\mathcal{S}_\varphi(t,\rho_{2}),\rho_{1}-\rho_{2}\rangle_{H}\leq c\left(1+\Vert\varphi(t)\Vert^2_{H}+\Vert\rho_{2}\Vert_{H^{1}}^2\right)\Vert\rho_{1}-\rho_{2}\Vert_{H}^2.}               		\end{array}
\end{equation}

Moreover, thanks to \eqref{sm108}, for every $\e>0$ there exists $c_\e>0$ such that for every $t\in[0,T]$ and $\rho\in H^{1}$
	\begin{equation} 
	\label{115}	
		\begin{aligned}
		\ds{\langle\mathcal{S}_\varphi(t,\rho),\rho\rangle_{H}}
		&\ds{=\langle\text{div}\left[b(\rho)\nabla\rho\right],\rho\rangle_{H}+\langle F_{g}(\rho),\rho\rangle_{H}+\langle\sigma_{g}(\rho)Q\varphi(t),\rho\rangle_{H} }\\
		\vs
		&\ds{ \leq -\frac{1}{\gamma_{1}}\Vert\rho\Vert_{H^{1}}^2-c\,\Vert \rho\Vert_{L^{\theta+1}}^{\theta+1}+c_\e\,\,\Vert \rho\Vert_H^2+c+\e\,\Vert\varphi(t)\Vert^2_{H}\Vert\rho\Vert_H    ^2.}
	\end{aligned}
	\end{equation}

Finally, according to \eqref{sm120} and \eqref{sm112}, for every $t \in\,[0,T]$ and $\rho \in\,H^1$ 
\begin{equation}
\label{sm121}	
\Vert\mathcal{S}(t,\rho)\Vert_{H^{-1}}\leq c\,\left(\Vert\rho\Vert_{H^{1}}+\Vert\rho\Vert_{L^{\theta+1}}^{\theta+1}+\Vert \rho\Vert_H^2+\Vert \varphi(t)\Vert_H^2\right).	
\end{equation}

For every $n$, we denote  $H_n:=\text{span}\left\{e_1,\ldots,e_n\right\}$ and we denote by $P_n$ the projection of $H$ onto $H_n$.
Next, we introduce the finite dimensional problem
\[u^\prime_n(t)=P_n S_\varphi(t,\rho_n(t)),\ \ \ \ \ u_n(0)=P_n \rho(0) \in\,H_n.\]
Since $P_n S_\varphi$ is quasi-monotone and coercive in $H_n$, there exists a unique solution $\rho_n \in\,C([0,T];H_n)$, such that $\rho_n^\prime \in\,L^2(0,T;H_n)$.

Now, we are showing that there exists $c_T>0$ such that for every $n \in\,\mathbb{N}$
\begin{equation}
	\label{sm125}
	\sup_{t \in\,[0,T]}\Vert \rho_n(t)\Vert_H^2+\int_0^T \Vert \rho_n(t)\Vert^{\theta+1}_{L^{\theta+1}}\,dt+\int_0^T \Vert S_\varphi(t,\rho_n(t))\Vert_{H^{-1}}^2\,dt\leq c_T.
\end{equation}

 Due to \eqref{115}, we have
\[\begin{array}{l}
\ds{\Vert \rho_n(t)\Vert_H^2=\Vert \rho_n(0)\Vert_H^2+\int_0^t\langle S_\varphi(s,\rho_n(s))\rangle_H\,ds\leq \Vert \rho_n(0)\Vert_H^2+c\,T}\\
\vs
\ds{ -\frac{1}{\gamma_{1}}\int_0^t\Vert\rho_n(s)\Vert_{H^{1}}^2\,ds-c\int_0^t\Vert \rho_n(s)\Vert_{L^{\theta+1}}^{\theta+1}\,ds+c_\e\,\,\int_0\Vert \rho_n(s)\Vert_H^2\,ds+\e\,\int_0^t\Vert\varphi(s)\Vert^2_{H}\Vert\rho_n(s)\Vert_H^2\,ds.}	
\end{array}
\]
Hence, if we take $\bar{\e}=(2\Vert \varphi\Vert_{L^2(0,T;H)})^{-1}$, we get
\[\begin{array}{l}
\ds{\frac 12 \sup_{r \in\,[0,t]}\Vert \rho_n(r)\Vert_H^2+\frac{1}{\gamma_{1}}\int_0^t\Vert\rho_n(s)\Vert_{H^{1}}^2\,ds+c\int_0^t\Vert \rho_n(s)\Vert_{L^{\theta+1}}^{\theta+1}\,ds}\\
\vs
\ds{\leq \Vert \rho_n(0)\Vert_H^2+c\,T+c_{\bar{\e}}\,\,\int_0\Vert \rho_n(s)\Vert_H^2\,ds,}
	\end{array}\]
	and the Gronwall lemma implies
\begin{equation}
\label{sm126}
\sup_{r \in\,[0,T]}\Vert \rho_n(r)\Vert_H^2+\int_0^T\Vert\rho_n(s)\Vert_{H^{1}}^2\,ds+\int_0^T\Vert \rho_n(s)\Vert_{L^{\theta+1}}^{\theta+1}\,ds	\leq c_T.	
\end{equation}
Finally, according to \eqref{sm121} and \eqref{sm126}, we have
\[\int_0^T \Vert S_\varphi(t,\rho_n(t))\Vert_{H^{-1}}^2\,dt\leq c\int_0^t\left(\Vert\rho_n(t)\Vert_{H^{1}}\,+\Vert\rho_n(t)\Vert_{L^{\theta+1}}^{\theta+1}+\Vert \rho_n(t)\Vert_H^2\right)\,dt+c\,\Vert \varphi\Vert_{L^2(0,T;H)}^2,\]
and this, together with \eqref{sm126}, implies \eqref{sm125}.

As a consequence of \eqref{sm125}, we have that there exists a subsequence, still denoted by $\rho_n$, and there exists $\rho \in\,L^2(0,T;H^1)$, with  $\partial_t \rho \in\,L^2(0,T;H^{-1})$,  and $\eta \in\,L^2(0,T;H^{-1})$ such that, as $n\to\infty$,
\[\rho_n\rightharpoonup \rho,\ \ \ \text{in}\ \ L^2(0,T;H^1),\ \ \ \ \ \partial_t\rho_n\rightharpoonup \partial_t\rho,\ \ \ \text{in}\ \ L^2(0,T;H^{-1}),\]
and
\[S_\varphi(\cdot,\rho_n)\rightharpoonup \eta\ \ \ \text{in}\ \ L^2(0,T;H^1).\]
Thanks to \eqref{sm114}, we can use the same arguments used in  \cite[Lemma 2.4]{liu} to show that  $\eta=S_\varphi(\cdot,u)$, as elements of $L^2(0,T;H^{-1})$. Moreover we can  prove  that $\partial_t \rho(t)=\eta(t)$. Therefore, we  conclude that $\rho$ is a weak solution of equation 
\eqref{limit-det}.

As far as uniqueness is concerned, as in \cite{liu} it is again a consequence of \eqref{sm114}.

\end{proof}

\begin{Corollary}
Assume Hypotheses \ref{Hypothesis1} and \ref{Hypothesis2} and either Hypothesis \ref{Hypothesis3-bis} or Hypothesis \ref{Hypothesis3} and fix any $T>0$, $u_0 \in\,H^1$ and $	\varphi \in\,L^2(0,T;H)$. Then, if we define $u:=g^{-1}(\rho)$, where $\rho$ is the unique weak solution of equation \eqref{limit-det}, we have that $u$  is the unique weak solution to equation \eqref{sm18} and 
\begin{equation}
\label{sm130}
u \in\,C([0,T];H)\cap L^2(0,T;H^1),\ \ \ \ \partial_t u \in\,L^2(0,T;H^{-1}).	
	\end{equation}
	\end{Corollary}

\begin{proof}
We define $u=g^{-1}(\rho)$, where $\rho$ is the unique solution to equation \eqref{limit-det}. Due to the fact that $\rho$ satisfies \eqref{sm110} and $g^{-1}$ is differentiable with bounded derivative, we have that 
\[u \in\,C([0,T];H)\cap L^2(0,T;H^1),\]
and
\[
\nabla (g(u(t)	)=g^\prime(u(t))\nabla u(t)=\gamma(u(t))\nabla u(t).
\]
Recalling how $b$, $F_g$ and $\si_g$ were defined, this implies
\[\begin{array}{l}
\ds{	S_\varphi(t,\rho(t))=\text{div}\left[b(\rho(t))\nabla\rho(t)\right]+F_g(\rho(t))+\si_g(\rho(t))Q\varphi(t)}\\
\vs
\ds{=\text{div}\left[\frac 1{\gamma(u(t))}\nabla (g(u(t))\right]+f(u(t))+\si(u(t))Q\varphi(t)=\Delta u(t)+f(u(t))+\si(u(t))Q\varphi(t)}
\end{array}
\]
in $H^{-1}$ sense.
Moreover, by mollfyng  $\rho$ with respect to $t$ and $x$ and then by taking the limit, we have that \[\partial_t u \in\,L^2(0,T;H^{-1})\] and 
\[\partial_t\rho(t)=g^\prime(u(t))\partial_t u(t)=\gamma(u(t))\partial_t u(t),\]
in $H^{-1}$ sense. We can now conclude, as we know that $\partial_t\rho(t)=S_\varphi(t,\rho(t))$, in $H^{-1}$ sense.
\end{proof}

\section{Proof of Theorem \ref{mainteo}}

In order to prove Theorem \ref{mainteo}, we will show that the conditions C1 and C2 that we introduced in Section \ref{sec3} are both satisfied. We will consider here the case the nonlinearity $f$ satisfies Hypothesis \ref{Hypothesis3} and we leave to the reader to adapt out proof to the case $f$ satisfies Hypothesis \ref{Hypothesis3-bis}.
\begin{Theorem}
Under Hypotheses \ref{Hypothesis1} and \ref{Hypothesis2} and either Hypothesis \ref{Hypothesis3-bis} or Hypothesis \ref{Hypothesis3}, condition C1 holds.	
\end{Theorem}

\begin{proof}
Let us fix $T, M>0$ and let $\{\varphi_\mu\}$ be a family of processes in $\Lambda_{T, M}$
such that 
\[\lim_{\mu\to 0} \varphi_\mu=\varphi,\ \ \ \ \text{in distribution in}\ L_w^2(0,T;H),\]
where $L^2_w(0,T;H)$ is the space $L^2(0,T;H)$ endowed with the weak topology and  $\varphi \in\,\Lambda_{T,M}$. 

For every sequence $\{\mu_k\}_{k \in\,\mathbb{N}}$ converging to $0$, as $k\uparrow\infty$, we denote $\rho_k=g(u_k)$, where $u_k:=u^{\varphi_{\mu_k}}_{\mu_k}$ is the solution of equation \eqref{SPDE-controlled}, corresponding to the control $\varphi_{\mu_k}$. Thanks to Theorem \ref{tightness} and Lemma \ref{lemma5.1}, we have that the family
\[\left\{\mathcal{L}(\rho_{k},\ \mu_k\partial_t u_{k},\varphi_{\mu_k})\right\}_{k \in\,\mathbb{N}} \subset \mathcal{P}\left(C([0,T];H^\d)\times C([0,T];H)\times\mathcal{S}_{T, M}\right)\]
is tight. We denote by $\rho$ a weak limit point for the sequence $\{\rho_k\}_{k \in\,\mathbb{N}}$ and we denote
\[\Lambda:=C([0,T];H^\d)\times C([0,T];H)\times\mathcal{S}_{T, M}\times C([0,T],U)),\]
where $U$ is the Hilbert space such that \eqref{contb} holds. 
By the Skorokhod Theorem there exist  random variables
\[ \mathcal{Y}=(\hat{\rho}, 0,\hat{\varphi}, \hat{w}^Q),\ \ \ \ \ \ \ \ 
\mathcal{Y}_k=\left( \hat{\rho}_k,\hat{\theta}_k,\hat{\varphi}_k,\hat{w}^Q_k  \right),\ \ \ \ k \in\,\mathbb{N},\]
defined on a probability space $(\hat{\Omega}, \hat{\mathcal{F}}, \{\hat{\mathcal{F}}_t\}_{t \in\,[0,T]}, \hat{\mathbb{P}})$, such that
\[\mathcal{L}(\mathcal{Y})=\mathcal{L}(\rho, 0, \varphi,w^Q),\ \ \ \ \ \ \ \ \ \ \mathcal{L}(\mathcal{Y}_k)=\mathcal{L}(\rho_{k},\ \mu_k\partial_t u_{k},\varphi_{\mu_k},w^Q),\ \ \ \ \ k \in\,\mathbb{N},\]
and such that
\begin{equation}
\label{sm130}	
\lim_{k\to\infty}\mathcal{Y}_k=\mathcal{Y}\ \ \ \ \text{in}\ \ \Lambda,\ \ \ \ \ \ \ \hat{\mathbb{P}}-\text{a.s.}
\end{equation}

For every $k \in\,\mathbb{N}$ and $\psi \in\,H^2$, we have
\[\begin{array}{l}
\ds{\langle \hat{\rho}_{k}(t)+\hat{\theta}_{k}(t),\psi\rangle_H= \langle g(u_{0})+\mu_k v_{0},\psi\rangle_H+\int_{0}^{t}\langle \text{div}\left[b(\hat{\rho}_k(s))\nabla\hat{\rho}_k(s)\right],\psi\rangle_H\,ds}\\
\vs
\ds{+\int_{0}^{t}\langle F_{g}(\hat{\rho}_k(s)),\psi\rangle_H\,ds+\int_{0}^{t}\langle \sigma_{g}(\hat{\rho}_k(s))Q\,\hat{\varphi}_k(s),\psi\rangle_H\,ds+\sqrt{\mu_k}\int_{0}^{t}\langle \sigma_{g}(\hat{\rho}_{\mu_k}(s))d\hat{w}_k^{Q}(s),\psi\rangle_H.}
	\end{array}\]

Thanks to \eqref{sm130}, for every $t \in\,[0,T]$  we have
\begin{equation}
\label{sm144}
\lim_{k\to \infty}\langle \hat{\rho}_{k}(t)+\hat{\theta}_{k}(t),\psi\rangle_H=\langle \hat{\rho}(t),\psi\rangle_H,\ \ \ \ \ \hat{\mathbb{P}}-\text{a.s.}	
\end{equation}

Next, if we define $\hat{u}_k:=g^{-1}(\hat{r}_k)$ and $\hat{u}:=g^{-1}(\hat{r})$, we have
\[\begin{array}{l}
    \ds{\int_0^t \langle b(\hat{\rho}_k(s)) \nabla \hat{\rho}_k(s),\nabla \psi \rangle_H ds -\int_0^t \langle b(\hat{\rho}(s)) \nabla \hat{\rho}(s),\nabla \psi \rangle_H \,ds}\\
    \vs
    \ds{=\int_0^t \langle  \nabla \hat{u}_k(s),\nabla \psi \rangle_H ds -\int_0^t \langle \nabla \hat{u}(s),\nabla \psi \rangle_H ds=-\int_0^t \langle  (\hat{u}_k(s)-\hat{u}(s)),\Delta \psi \rangle_H ds.}
\end{array}\]
In particular, since \eqref{sm130} implies the $\mathbb{P}$-a.s. convergence of $\hat{u}_k$ to $\hat{u}$ in $C([0,T];H)$, we get that 
\begin{equation}\label{eq:gs68}
    \lim_{k\to \infty} \int_0^t \langle b(\hat{\rho}_k(s)) \nabla \hat{\rho}_k(s),\nabla \psi \rangle_H\, ds =\int_0^t \langle b(\hat{\rho}(s)) \nabla \hat{\rho}(s),\nabla \psi \rangle_H \, ds,\ \ \ \ \hat{\mathbb{P}}-\text{a.s.}
    \end{equation}

It is immediate to check that estimate \eqref{sm15-l1} extends to $F_g$. Thus we have
\[\begin{array}{l}
\ds{\left|\int_0^t\langle F_{g}(\hat{\rho}_k(s)),\psi\rangle_H\,ds-\int_0^t\langle F_{g}(\hat{\rho}(s)),\psi\rangle_H\,ds\right|\leq \int_0^t\Vert F_{g}(\hat{\rho}_k(s))-F_{g}(\hat{\rho}(s))\Vert_{L^1}\,ds\,\Vert \psi\Vert_{H^1}}	\\
\vs\ds{\leq c\int_0^t \left(1+\Vert \hat{\rho}_k(s)\Vert_{L^{2(\theta-1)}}^{\theta-1}+ \Vert \hat{\rho}(s)\Vert_{L^{2(\theta-1)}}^{\theta-1}\right)\Vert \hat{\rho}_k(s)-\hat{\rho}(s)\Vert_H\,ds\,\Vert \psi\Vert_{H^1},}
\end{array}\]
so that, thanks to \eqref{sm130}, 
\begin{equation}
\label{sm140}
\lim_{k\to\infty}\int_0^t\langle F_{g}(\hat{\rho}_k(s)),\psi\rangle_H\,ds=\int_0^t\langle F_{g}(\hat{\rho}(s)),\psi\rangle_H\,ds,\ \ \ \ \ \hat{\mathbb{P}	}-\text{a.s.}
\end{equation}

Now,  for any $h\in L^{2}(0,T;H)$, 
	\begin{align*}
		\left|\int_{0}^{t}\langle\sigma_{g}(\hat{\rho}(s))Qh(s),\psi\rangle_{H} ds\right|
		\leq c\Vert\psi\Vert_{H}\left(\int_{0}^{T}(1+\Vert\hat{\rho}(s)\Vert_{H}^{2})ds\right)^{\frac{1}{2}}\left(\int_{0}^{T}\Vert h(s)\Vert_{H}^{2}ds\right)^{\frac{1}{2}}
	\end{align*}
and this implies that the mapping
		\begin{equation*}
		h \in\,L^2(0,t;H)\mapsto \int_{0}^{t}\langle\sigma_{g}(\hat{\rho}(s)))Qh(s),\psi\rangle_{H}\, ds \in\,\mathbb{R}
	\end{equation*}
	is a linear functional so that, thanks to \eqref{sm130}
	\begin{equation}
	\label{sn141}
	\lim_{k\to\infty}\int_{0}^{t}\langle\sigma_{g}(\hat{\rho}(s))Q\hat{\varphi}_k(s),\psi\rangle_{H} ds=\int_{0}^{t}\langle\sigma_{g}(\hat{\rho}(s))Q\hat{\varphi}(s),\psi\rangle_{H}\, ds,\ \ \ \ \ \hat{\mathbb{P}}-\text{a.s.}	
	\end{equation}
Moreover, we have
\[\begin{array}{l}
\ds{\left|\int_{0}^{t}\langle(\sigma_{g}(\hat{\rho}_k(s))-\sigma_{g}(\hat{\rho}(s)))Q\hat{\varphi}_k(s),\psi\rangle_{H} ds\right|}\\
\vs
\ds{\leq \Vert \psi\Vert_H\int_0^t\Vert \sigma_{g}(\hat{\rho}_k(s))-\sigma_{g}(\hat{\rho}(s))\Vert_{\mathcal{L}(H_Q,H)}\Vert\hat{\varphi}_k(s)\Vert_H\,ds\leq c\,\Vert \psi\Vert_H \Vert \hat{\rho}_k-\hat{\rho}\Vert_{L^2(0,T;H)}\Vert\hat{\varphi}_k\Vert_{L^2(0,T;H)},}	
\end{array}\]
and by using again \eqref{sm130} we get
\[\lim_{k\to\infty} \int_{0}^{t}\langle(\sigma_{g}(\hat{\rho}_k(s))-\sigma_{g}(\hat{\rho}(s)))Q\hat{\varphi}_k(s),\psi\rangle_{H} ds=0,\ \ \ \ \ \ \ \hat{\mathbb{P}}-\text{a.s.}\]
This, together with \eqref{sn141}, implies
\begin{equation}
\label{sm143}
\lim_{k\to\infty}\int_{0}^{t}\langle\sigma_{g}(\hat{\rho}_k(s))Q\hat{\varphi}_k(s),\psi\rangle_{H} ds=	\int_{0}^{t}\langle\sigma_{g}(\hat{\rho}(s))Q\hat{\varphi}(s),\psi\rangle_{H} ds,\ \ \ \ \ \ \ \ \hat{\mathbb{P}}-\text{a.s.}
\end{equation}

Finally, since
\[\sup_{k \in\,\mathbb{N}}\,\hat{\mathbb{E}}\sup_{t \in\,[0,T]}	\left|\int_{0}^{t}\langle \sigma_{g}(\hat{\rho}_{\mu_k}(s))d\hat{w}_k^{Q}(s),\psi\rangle_H\right|	^2\leq c\,\sup_{k \in\,\mathbb{N}}\,\Vert\psi\Vert_H^2\int_0^T\left(1+\hat{\mathbb{E}}\Vert \hat{\rho}_{\mu_k}(s)\Vert_H^2\right)\,ds<\infty,\]
we conclude that
\[\lim_{k\to \infty}\sqrt{\mu_k}\,\sup_{t \in\,[0,T]}	\hat{\mathbb{E}}\left|\int_{0}^{t}\langle \sigma_{g}(\hat{\rho}_{\mu_k}(s))d\hat{w}_k^{Q}(s),\psi\rangle_H\right|	^2=0.\]

This, together with \eqref{sm144}, \eqref{eq:gs68}, \eqref{sm140} and \eqref{sm143} implies that 
\[\begin{array}{ll}
			\ds{\langle\hat{\rho}(t),\psi\rangle_{H}=} & \ds{\langle g(u_{0}),\psi\rangle_H-\int_{0}^{t}\langle b(\hat{\rho}(s))\nabla\hat{\rho}(s),\nabla\psi\rangle_{H}ds}\\
			\vs
			&\ds{\int_{0}^{t}\left[\langle F_{g}(\hat{\rho}(s)),\psi\rangle_{H}+\langle\sigma_{g}(\hat{\rho}(s))Q\hat{\varphi}(s),\psi\rangle_{H}\right]\,ds.}
		\end{array}\]
		As proven in Theorem \ref{teo-det}, the equation above has a unique solution. Thus,  for every sequence $\{\mu_k\}_{k \in\,\mathbb{N}}\downarrow 0$ the sequence $\{\rho_k\}_{k \in\,\mathbb{N}}$ converges in distribution to the solution $\rho$ of equation \eqref{limit-det} with respect to the strong topology of $C([0,T];H^\d)$, for any arbitrary $\d<1$, and hence with respect to the strong topology of $C([0,T];L^p(\mathcal{O}))$, for every $p<\infty$, if $d=1,2$ and $p<2d/(d-2)$, if $d>2$.
	In particular, this implies that the sequence $\{u_k\}_{k \in\,\mathbb{N}}$ converges in distribution to the solution $u^\varphi$ of equation \eqref{sm18} with respect to the strong topology  of $C([0,T];L^p(\mathcal{O}))$ and condition C1 holds.
	
	As a consequence of the arguments used above to prove condition C1, we have that the mapping
	\[\varphi \in\,L^2_w(0,T;H)\mapsto u^\varphi \in\,C([0,T];L^p(\mathcal{O}))\]
	is continuous. Therefore, since $\Lambda_{T,M}$ is compact in $L^2_w(0,T;H)$, for every $M>0$, we have that 
	\[\Phi_{T,R}=\left\{I_T\leq R\right\}=\{ u^\varphi\ :\ \varphi \in\,\Lambda_{T,2R^2}\} \]
	is compact, and condition C2 follows.
\end{proof}

%%%%%%%%%%%%%%%%%%%%%%%%%%%%%%%%%%%%%

\appendix

\section{The small mass limit for system \eqref{SPDE1}}
\label{A1}

The Smoluchowski-Kramers approximation for system \eqref{SPDE1} has been studied in \cite{CX}, in the case  $f$ is Lipschitz-continuous and  $\sigma$ is bounded. Here, we prove  an analogous result when $\sigma$ is unbounded and $f$ has polynomial growth (see Hypothesis \ref{Hypothesis3}).

In what follows, we shall assume that  Hypotheses \ref{Hypothesis1}, \ref{Hypothesis2} and Hypothesis \ref{Hypothesis3} hold. By applying Theorem \ref{well-posedness} with the control $\varphi=0$, we have that for every $T>0$ and every $(u_{0},v_{0})\in\H_{1}$ there exists a unique adapted solution  $u_\mu$ to equation \eqref{SPDE1}, such that $(u_{\mu},\partial_{t}u_{\mu})\in L^{2}(\Omega;C([0,T];\H_{1}))$.
Now, we take $\theta\in(1,3)$ and for every $a\in[0,1)$ and $\delta>0$ we denote 
\begin{equation*}
	X_{1}(a):=\bigcap_{q<q(a)}L^{q}(0,T;H^{a}),\ \ \ X_{2}(\delta):=\bigcap_{p<\theta+1}L^{p}(0,T;H^{-\delta}),\ \ \ X_{3}(a):=\bigcap_{p<2/a}L^{p}(0,T;H^{a}),
\end{equation*} 
where 
\begin{equation*}
	q(a):=\frac{2(\theta+1)}{2+(\theta-1)a}.
\end{equation*}

\begin{thm}\label{sm}
	Assume Hypotheses \ref{Hypothesis1}, \ref{Hypothesis2} and \ref{Hypothesis3} hold, and assume $\theta\in(1,3)$. For every $\mu>0$, let $u_{\mu}$ denote the unique solution to equation \eqref{SPDE1}, with the  initial conditions $(u_{0},v_{0})\in\mathcal{H}_{1}$.
	
	\begin{enumerate}
	
	\item For every $a\in[0,1)$ and $\delta>0$, and for every $\eta>0$ we have
	\begin{equation*}
		\lim_{\mu\to0}\mathbb{P}\Big(\norm{u_{\mu}-u}_{X_{1}(a)}+\norm{u_{\mu}-u}_{X_{2}(\delta)}>\eta\Big)=0,
	\end{equation*}  
	
	where $u\in L^{2}(\Omega;L^{2}([0,T];H^{1}))$ is the unique solution to equation \eqref{SPDE3}, with initial datum $u_{0}$.\\
	
	\item If we assume also that
	\begin{equation}\label{sigma_poly}
		\norm{\sigma(h)}_{\mathcal{L}_{2}(H_{Q},H)}\leq c\big(1+\norm{h}_{H}^{\rho}\big),\ \ \forall h\in H,
	\end{equation}

	for some $\rho\in[0,(\theta+1)/4)$, then for every $a\in[0,1)$ we have
	\begin{equation*}
		\lim_{\mu\to0}\mathbb{P}\Big(\norm{u_{\mu}-u}_{X_{3}(a)}>\eta\Big)=0.
	\end{equation*}  
	
	\end{enumerate}
\end{thm}

\bigskip

%%%%%%%%%%%%%%%%%%%%%%%%%%%%%%%%%%%

\subsection{Energy estimates}
One of the key ingredients in our proof of  Theorem \ref{sm}  is the tightness of $\{u_\mu\}_{\mu \in\,(0,1)}$ in suitable functional spaces. This will require the following a-priori bounds.

\begin{lem}\label{Energy}
	Assume Hypotheses \ref{Hypothesis1}, \ref{Hypothesis2} and \ref{Hypothesis3} hold,  with $\theta\in(1,3)$, and fix $T>0$ and $(u_{0},v_{0})\in\mathcal{H}_{1}$.
	\begin{enumerate}
	\item There exist $\mu_{T}\in(0,1)$ and $c_{T}>0$ such that for every $\mu\in(0,\mu_{T})$,
	\begin{equation}\label{est1}
		\mathbb{E}\sup_{t\in[0,T]}\norm{u_{\mu}(t)}_{H}^{2}+\mathbb{E}\int_{0}^{T}\norm{u_{\mu}(t)}_{H^{1}}^{2}dt+\mathbb{E}\int_{0}^{T}\norm{u_{\mu}(t)}_{L^{\theta+1}}^{\theta+1}dt\leq c_{T},
	\end{equation}
	
	and
	\begin{equation}\label{est2}
		\mathbb{E}\sup_{t\in[0,T]}\norm{u_{\mu}(t)}_{H^{1}}^{2}+\mathbb{E}\sup_{t\in[0,T]}\norm{u_{\mu}(t)}_{L^{\theta+1}}^{\theta+1}+\mu\mathbb{E}\sup_{t\in[0,T]}\norm{\partial_{t}u_{\mu}(t)}_{H}^{2}+\mathbb{E}\int_{0}^{T}\norm{\partial_{t}u_{\mu}(t)}_{H}^{2}dt\leq \frac{c_{T}}{\mu}.
	\end{equation}

	\item If, in addition,  condition \eqref{sigma_poly} holds, then for every $\mu\in(0,\mu_{T})$
	\begin{equation}\label{est3}
		\mathbb{E}\sup_{t\in[0,T]}\norm{u_{\mu}(t)}_{H^{1}}^{2}+\mathbb{E}\sup_{t\in[0,T]}\norm{u_{\mu}(t)}_{L^{\theta+1}}^{\theta+1}+\mu\mathbb{E}\sup_{t\in[0,T]}\norm{\partial_{t}u_{\mu}(t)}_{H}^{2}\leq \frac{c_{T}}{\mu^{\beta}},
	\end{equation}
	
	where 
	\begin{equation*}
		\beta=\beta(\rho):=\frac{\theta+1}{2(\theta+1-2\rho)}.
	\end{equation*}
	
	\end{enumerate}
\end{lem}

\begin{Remark}
	{\em   \begin{enumerate}
	
	\item If the mapping $\sigma$ is bounded, then  condition \eqref{sigma_poly} is satisfied for $\rho=0$, in which case $\beta=1/2$, so that for every $\mu\in(0,\mu_{T})$
	\begin{equation*}
		\sqrt{\mu}\ \mathbb{E}\sup_{t\in[0,T]}\Bigg(\norm{u_{\mu}(t)}_{H^{1}}^{2}+\norm{u_{\mu}(t)}_{L^{\theta+1}}^{\theta+1}+\mu\norm{\partial_{t}u_{\mu}(t)}_{H}^{2}\Bigg)\leq c_{T}.
	\end{equation*}
	
	\item  If $\rho\in[0,(\theta+1)/4)$, then   $\beta<1$. Therefore, if condition (\ref{sigma_poly}) holds,  due to (\ref{est3}) we have 
	\begin{equation}\label{est4}
		\lim_{\mu\to0}\mu^{2}\mathbb{E}\sup_{t\in[0,T]}\norm{\partial_{t}u_{\mu}(t)}_{H}^{2}=0,
	\end{equation}
	which is the same bound proven in \cite{CX}.
	
	\end{enumerate}}
\end{Remark}

\begin{proof} Estimates (\ref{est1}) and (\ref{est2}) can be proved by proceeding as in the proof of Lemma \ref{lemma5.1}. Thus, we will only prove (\ref{est3}) under  condition (\ref{sigma_poly}).

	For every $\mu\in(0,1)$ and $t\in[0,T]$, define 
	\begin{equation*}
		L_{\mu}(t):=\norm{u_{\mu}(t)}_{H^{1}}^{2}+\int_0^L\Big(c_{2}-\mathfrak{f}(x,u_{\mu}(t,x))\Big)dx+\mu\norm{\partial_{t}u_{\mu}(t)}_{H}^{2},
	\end{equation*}
	where the function $\mathfrak{f}$ and the constant $c_2$ have been introduced in Hypothesis \ref{Hypothesis3}.
	Due to  \eqref{sm12}, we have that
	\[L_{\mu}(t)\geq \Vert u_\mu(t)\Vert^2_{H^1}+\Vert u_\mu(t)\Vert_{L^{\theta+1}}^{\theta+1}+\mu\,\Vert \partial_t u_\mu(t)\Vert^2_{H},\ \ \ \ \  \mathbb{P}-\text{a.s.}\] 
	Thus, we obtain 	\eqref{est3} once we have proved that
	\begin{equation}\label{claim}
		\mu^{\beta}\mathbb{E}\sup_{t\in[0,T]}L_{\mu}(t)\leq c_{T}.
	\end{equation}
	
	Assume (\ref{claim}) is not true. Then there is a sequence $(\mu_{k})_{k\in\mathbb{N}}\subset(0,1)$ converging to $0$, as $k\to\infty$, such that
	\begin{equation}\label{contra}
		\lim_{k\to\infty}\mu_{k}^{\beta}\,\mathbb{E}\sup_{t\in[0,T]}L_{\mu_{k}}(t)=+\infty.
	\end{equation}
	For every $k\in\mathbb{N}$, the mapping $t\mapsto L_{\mu_k}(t)$ is  continuous $\mathbb{P}$-a.s., so that  there exists a random time $t_{k}\in[0,T]$ such that
	\begin{equation*}
		L_{\mu_{k}}(t_{k})=\sup_{t\in[0,T]}L_{\mu_{k}}(t).
	\end{equation*}
	As a consequence of the It\^o formula,  if $s$ is any random time  such that $\mathbb{P}(s\leq t_{k})=1$,  we have 
	\begin{equation*}
		L_{\mu_{k}}(t_{k})-L_{\mu_{k}}(s)\leq \frac{1}{\mu_k}\int_{s}^{t_{k}}\norm{\sigma(u_{\mu_{k}}(r))}_{\mathcal{L}_{2}(H_{Q},H)}^{2}dr+2\big(M_{k}(t_{k})-M_{k}(s)\big),
	\end{equation*}
	where 
	\begin{equation*}
		M_{k}(t):=\int_{0}^{t}\Inner{\partial_{t}u_{\mu_{k}}(r),\sigma(u_{\mu_{k}}(r))dw^{Q}(r)}_{H}.
	\end{equation*}
Thanks to Young's inequality, since $2\beta>1$ and $\mu_k<1$,  we have
	\begin{equation*}
		\begin{array}{l}
			\ds{\frac{1}{\mu_k}\int_{s}^{t_{k}}\norm{\sigma(u_{\mu_{k}}(r))}_{\mathcal{L}_{2}(H_{Q},H)}^{2}dr
			\leq \frac{c}{\mu_k}\int_{s}^{t_{k}}\Big(1+\norm{u_{\mu_{k}}(r)}_{H}^{2\rho}\Big)dr}\\
			\vs
			\ds{\leq c\Bigg(\frac{t_{k}-s}{\mu_k}+\frac{t_{k}-s}{\mu_{k}^{\frac{\theta+1}{\theta+1-2\rho}}}+\int_{0}^{T}\norm{u_{\mu_{k}}(t)}_{H}^{\theta+1}dt\Bigg)\leq c\left(\frac{t_{k}-s}{\mu^{2\beta}}+\int_{0}^{T}\norm{u_{\mu_{k}}(t)}_{H}^{\theta+1}dt\right).}	
		\end{array}
	\end{equation*}	
Therefore, 	
if we define 
	\begin{equation*}
		U_{k}:=\int_{0}^{T}\norm{u_{\mu_{k}}(t)}_{L^{\theta+1}}^{\theta+1}dt,\ \ \ \text{and}\ \ \ M_{k}:=\sup_{t\in[0,T]}\abs{M_{k}(t)},
	\end{equation*}
we can fix a constant $\kappa_{T}$ independent of $k$, with $L_{\mu_k}(0)\leq \kappa_{T}$, 
	 such that
	\begin{equation*}
		L_{\mu_{k}}(t_{k})-L_{\mu_{k}}(s)\leq \kappa_{T}\Big(\frac{t_{k}-s}{\mu_{k}^{2\beta}}+U_{k}\Big)+4M_{k}.
	\end{equation*}
	In particular, if we take $s=0$, 
	we have 
	\begin{equation*}
		t_{k}\geq \frac{\mu_{k}^{2\beta}}{\kappa_{T}}\Big(L_{\mu_{k}}(t_{k})-4M_{k}-\kappa_{T}(U_{k}+1)\Big)=:\frac{\mu_{k}^{2\beta}}{\kappa_{T}}\delta_{k}.
	\end{equation*}
	On the set $E_{k}:=\big\{\delta_{k}>0\big\}$, for any $s\in\big[t_{k}-\frac{\mu_{k}^{2\beta}}{2\kappa_{T}}\delta_{k},t_{k}\big]$ we have 
	\begin{equation*}
		L_{\mu_{k}}(s)\geq L_{\mu_{k}}(t_{k})-\frac{1}{2}\delta_{k}-\kappa_{T}U_{k}-4M_{k}=\frac{1}{2}\Big[L_{\mu_{k}}(t_{k})-4M_{k}-\kappa_{T}(U_{k}-1)\Big].
	\end{equation*}
	Hence, if we define 
	\begin{equation*}
		I_{k}:=\int_{0}^{T}L_{\mu_{k}}(s)ds,
	\end{equation*}
	we have
	\begin{equation*}
		I_{k}\geq \int_{t_{k}-\frac{\mu_{k}^{2\beta}}{2\kappa_{T}}\delta_{k}}^{t_{k}}L_{\mu_{k}}(s)ds\geq \frac{\mu_{k}^{2\beta}}{4\kappa_{T}}\Big[\Big(L_{\mu_{k}}(t_{k})-4M_{k}-\kappa_{T}U_{k}\Big)^{2}-\kappa_{T}^{2}\Big],
	\end{equation*}	
so that
	\begin{equation}\label{contra1}
		\mathbb{E}(I_{k})\geq \mathbb{E}(I_{k};E_{k})\geq \mathbb{E}\Bigg(\frac{\mu_{k}^{2\beta}}{4C_{T}}\Big(L_{\mu_{k}}(t_{k})-4M_{k}-C_{T}U_{k}\Big)^{2};E_{k}\Bigg)-\frac{\mu_{k}^{2\beta}}{4}C_{T}.
	\end{equation}
	
Now, by (\ref{est1}) and (\ref{est2}), we have $\sup_{k}\mathbb{E}(U_{k})<\infty$. Moreover, thanks to condition \eqref{sigma_poly}, and estimates  (\ref{est1}) and (\ref{est2}) we have
	\begin{equation*}
		\begin{array}{ll}
			\ds{\mathbb{E}(M_{k})} & \ds{\leq \mathbb{E}\Bigg(\int_{0}^{T}\norm{\partial_{t}u_{\mu_{k}}(t)}_{H}^{2}\norm{\sigma(u_{\mu_{k}}(t))}_{\mathcal{L}_{2}(H_{Q},H)}^{2}dt\Bigg)^{\frac{1}{2}}}\\
			\vs
			&\ds{\leq \mathbb{E}\,\sup_{t\in[0,T]}\norm{\sigma(u_{\mu_{k}}(t))}_{\mathcal{L}_{2}(H_{Q},H)}\left(\int_{0}^{T}\norm{\partial_{t}u_{\mu_{k}}(t)}_{H}^{2}dt\right)^{\frac{1}{2}}}\\
			\vs
			&\ds{\leq c\,\mathbb{E}\sup_{t\in[0,T]}\norm{\sigma(u_{\mu_{k}}(t))}_{\mathcal{L}_{2}(H_{Q},H)}^{\frac{2}{\rho}}+c\,\mathbb{E}\Big(\int_{0}^{T}\norm{\partial_{t}u_{\mu_{k}}(t)}_{H}^{2}dt\Big)^{\frac{1}{2-\rho}}}\\
			\vs
			&\ds{\leq c\,\Big(1+\mathbb{E}\sup_{t\in[0,T]}\norm{u_{\mu_{k}}(t)}_{H}^{2}\Big)+c\Big(\int_{0}^{T}\mathbb{E}\norm{\partial_{t}u_{\mu_{k}}(t)}_{H}^{2}dt\Big)^{\frac{1}{2-\rho}}\leq \frac{C_{T}}{\mu^{\frac{1}{2-\rho}}}.}
		\end{array}
	\end{equation*}
In particular, since
	\begin{equation*}
		\beta=\frac{\theta+1}{2(\theta+1-2\rho)}\geq \frac{1}{2-\rho},
	\end{equation*}
we have that
	\begin{equation*}
		\lim_{k\to\infty}\mu_{k}^{\beta}\mathbb{E}(M_{k})<+\infty,
	\end{equation*}
and thanks to (\ref{contra}) this gives	\begin{equation}
	\label{contra2}	\lim_{k\to\infty}\mu_{k}^{\beta}\mathbb{E}(\delta_{k})=+\infty.
	\end{equation}

Now, we have
	\[
		\mu_{k}^{\beta}\mathbb{E}(\delta_{k})=\mu_{k}^{\beta}\mathbb{E}(\delta_{k};E_{k})\leq \mathbb{E}\big(\mu_{k}^{\beta}(\delta_{k}+\kappa_{T});E_{k}\big)\leq \Big[\mathbb{E}\big(\mu_{k}^{2\beta}(\delta_{k}+\kappa_{T})^{2};E_{k}\big)\Big]^{\frac{1}{2}},
	\]
so that, thanks to (\ref{contra1}) we have
	\begin{equation*}
		\mathbb{E}(I_{k})\geq \frac{1}{4\kappa_{T}}\mathbb{E}\big(\mu_{k}^{2\beta}(\delta_{k}+\kappa_{T})^{2};E_{k}\big)-\frac{\mu_{k}^{2\beta}\kappa_{T}}{4}\geq \frac{1}{4\kappa_{T}}\big[\mu_{k}^{\beta}\mathbb{E}(\delta_{k})\big]^{2}-\frac{\mu_{k}^{2\beta}\kappa_{T}}{4}.
	\end{equation*}
Due to \eqref{contra2}, this implies
\[\lim\limits_{k\to\infty}\mathbb{E}(I_{k})=+\infty.\]	

However, the limit above is not possible. Actually, since we 	\begin{equation*}
		L_{\mu_{k}}(t)\leq \norm{u_{\mu_{k}}(t)}_{H^{1}}^{2}+\mu_{k}\norm{\partial_{t}u_{\mu_{k}}(t)}_{H}^{2}+L\,c\Big(1+\norm{u_{\mu_{k}}(t)}_{L^{\theta+1}}^{\theta+1}\Big),\ \ \mathbb{P}\text{-a.s.}
	\end{equation*}
	as a consequence of (\ref{est1}) and (\ref{est2}) we have
	\[ \sup_{k \in\,\mathbb{N}}\mathbb{E}(I_{k})<+\infty,\] and this gives  a contradiction. In particular, this means that claim (\ref{claim}) is true, and  (\ref{est3}) holds.
	
\end{proof}

%%%%%%%%%%%%%%%%%%%%%%%%%%%%%%%%%%%%

\subsection{Tightness}
As in Section \ref{sec5}, for every $T>0$ and $\mu>0$ we have defined
\begin{equation*}
	\rho_{\mu}(t,x)=g(u_{\mu}(t,x)),\ \ \ (t,x)\in[0,T]\times [0,L],
\end{equation*}
and, by integrating  equation (\ref{SPDE1}) with respect to $t$, we got
\begin{equation}\label{sto5}
	\begin{aligned}
		\rho_{\mu}(t)+\mu\partial_{t}u_{\mu}(t)
		& = g(u_{0})+\mu v_{0}+\int_{0}^{t}\text{div}\big[b(\rho_{\mu}(s))\nabla\rho_{\mu}(s)\big]ds+\int_{0}^{t}F_{g}(\rho_{\mu}(s))ds\\
		&\ \ \ \ +\int_{0}^{t}\sigma_{g}(\rho_{\mu}(s))dw^{Q}(s),
	\end{aligned}
\end{equation}
where we recall that 
$b$, $f_{g}$, $F_g$ and $\sigma_{g}$ were defined in \eqref{end1} and \eqref{end2}.

\begin{Definition}
	Let $E$ be a Banach space with norm $\norm{\cdot}_{E}$. Given $r>1$ and $\lambda\in(0,1)$, we denote by $W^{\lambda,r}(0,T;E)$ be the Banach space of all $u\in L^{p}(0,T;E)$ such that
	\begin{equation*}
		[u]_{W^{\lambda, r}(0,T;E)}:=\int_{0}^{T}\int_{0}^{T}\frac{\norm{u(t)-u(s)}_{E}^{r}}{\abs{t-s}^{1+\lambda r}}dtds<\infty,
	\end{equation*}
	endowed with the norm
	\begin{equation*}
		\norm{u}_{W^{\lambda,r}(0,T;E)}^{r}=\int_{0}^{T}\norm{u(t)}_{E}^{r}dt+[u]_{W^{\lambda, r}(0,T;E)}.
	\end{equation*}
\end{Definition}

It is possible to prove that if $\lambda r<1$, $p\leq r/(1-\lambda r)$ and $1\leq r\leq p$, then $W^{\lambda, r}(0,T;E)\subset L^p(0,T;E)$ and there exists some $c>0$ such that for all $u \in\,W^{\lambda, r}(0,T;E)$
\begin{equation}  \label{end7}
\Vert \tau_h(u)-u\Vert_{L^{p}(0,T-h;E)}\leq c\,h^\lambda T^{1/p-1/r}\,[u]_{W^{\la, r}(0,T;E)},\ \ \ \ h>0,	
\end{equation}
where
\[\tau_h(u)(t)=u(t+h),\ \ \ \ t \in\,[-h,T-h],\]
(see \cite[Lemma5]{simon1986}).

\begin{prop}\label{sm_tightness}
	Assume Hypotheses \ref{Hypothesis1}, \ref{Hypothesis2} and \ref{Hypothesis3} hold, with $\theta\in(1,3)$. Fix any $T>0$ and $(u_{0},v_{0})\in\mathcal{H}_{1}$ and let $(\mu_{k})_{k\in\mathbb{N}}\subset(0,1)$ be an arbitrary sequence converging to $0$.
	\begin{enumerate}
	\item The family of probability measures $\big(\mathcal{L}(\rho_{\mu_{k}})\big)_{k\in\mathbb{N}}$ is tight in $X_{1}(a)\cap X_{2}(\delta)$, for every $a\in[0,1)$ and $\delta>0$.
	
	\item If   condition (\ref{sigma_poly}) holds, then the family of probability measures  $\big(\mathcal{L}(\rho_{\mu_{k}})\big)_{k\in\mathbb{N}}$ is tight in $X_{3}(a)$, for every $a\in[0,1)$.
	
	\end{enumerate}
\end{prop}

\medskip

\begin{Remark}
	{\em \begin{enumerate}
	
	\item By taking $a=0$, we have 
	\[\left(\mathcal{L}(\rho_{\mu_{k}})\right)_{k\in\mathbb{N}}\ \text{ is tight in} \ \bigcap_{q<\theta+1}L^{q}(0,T;H).\]
	and thanks to the embedding $H^{a}\hookrightarrow C([0,L])$, for $a\in(1/2,1)$, we have
	\begin{equation*}
		\big(\mathcal{L}(\rho_{\mu_{k}})\big)_{k\in\mathbb{N}}\ \ \text{is tight in}\ \bigcap_{q<4(\theta+1)/(\theta+3)}L^{q}(0,T;C([0,L])).
	\end{equation*}
	
	\item When condition (\ref{sigma_poly}) holds, we have that
	\begin{equation*} 
		\big(\mathcal{L}(\rho_{\mu_{k}})\big)_{k\in\mathbb{N}}\ \ \text{is tight in}\ \bigcap_{p<\infty}L^{p}(0,T;H).
	\end{equation*}

	\end{enumerate}}
\end{Remark}

\begin{proof} For every $0\leq t_{1}\leq t_{2}\leq T$, we have
	\begin{equation*}
		\begin{aligned}
			\mathbb{E}\norm{\int_{t_{1}}^{t_{2}}\text{div}\big[b(\rho_{\mu}(s))\nabla\rho_{\mu}(s)\big]ds}_{H^{-1}}^{(\theta+1)/\theta}
			&\leq \Bigg(\mathbb{E}\norm{\int_{t_{1}}^{t_{2}}\text{div}\big[b(\rho_{\mu}(s))\nabla\rho_{\mu}(s)\big]ds}_{H^{-1}}^{2}\Bigg)^{(\theta+1)/2\theta}\\
			&\leq C(t_{2}-t_{1})^{(\theta+1)/2\theta}\Bigg(\mathbb{E}\int_{0}^{T}\norm{\rho_{\mu}(s)}_{H^{1}}^{2}ds\Bigg)^{(\theta+1)/2\theta},
		\end{aligned}
	\end{equation*}

	\begin{equation*}
		\begin{aligned}
			\mathbb{E}\norm{\int_{t_{1}}^{t_{2}}F_{g}(\rho_{\mu_{k}}(s))ds}_{H^{-1}}^{(\theta+1)/\theta}
			&\leq C\mathbb{E}\Bigg(\int_{t_{1}}^{t_{2}}\Big(1+\norm{u_{\mu}(s)}_{\theta+1}^{\theta}\Big)ds\Bigg)^{(\theta+1)/\theta}\\
			&\leq C(t_{2}-t_{1})^{1/\theta}\mathbb{E}\int_{0}^{T}\Big(1+\norm{u_{\mu}(s)}_{\theta+1}^{\theta+1}\Big)ds,
		\end{aligned}
	\end{equation*}
	and
	\begin{equation*}
		\begin{aligned}
			\mathbb{E}\norm{\int_{t_{1}}^{t_{2}}\sigma_{g}(\rho_{\mu}(s))dw^{Q}(s)}_{H^{-1}}^{(\theta+1)/\theta}
			&\leq \Bigg(\mathbb{E}\norm{\int_{t_{1}}^{t_{2}}\sigma_{g}(\rho_{\mu}(s))dw^{Q}(s)}_{H^{-1}}^{2}\Bigg)^{(\theta+1)/2\theta}\\
			&\leq C(t_{2}-t_{1})^{(\theta+1)/2\theta}\Bigg(1+\mathbb{E}\sup_{t\in[0,T]}\norm{u_{\mu}(t)}_{H}^{2}\Bigg)^{(\theta+1)/2\theta}.
		\end{aligned}
	\end{equation*}
	In view of \eqref{est1} and \eqref{sto5}, it is not difficult to show  that for every $\lambda\in(0,1/(\theta+1))$, 
	\begin{equation}\label{end4}
		\sup_{\mu\in(0,\mu_{T})}\mathbb{E}[\rho_{\mu}+\mu\partial_{t}u_{\mu}]_{W^{\lambda,\theta_{0}}(0,T;H^{-1})}^{\theta_{0}}<\infty,
	\end{equation}
	where 
	\begin{equation*}
		\theta_{0}:=\frac{\theta+1}{\theta}\in(1,2).
	\end{equation*}
	Moreover, by \eqref{est1} and \eqref{est2} we have 
	\begin{equation}\label{end5}
		\sup_{\mu\in(0,\mu_{T})}\mathbb{E}\norm{\rho_{\mu}+\mu\partial_{t}u_{\mu}}_{L^{\infty}([0,T];H)}^{2}<\infty.
	\end{equation}
	Therefore, from \eqref{end4} and \eqref{end5} we conclude that  for every $\epsilon>0$, there exists $L_{1}(\epsilon)>0$ such that, if we define
	\begin{equation*}
		K_{1}^{\epsilon}=\left\{f:[0,T]\times\mathbb{R}\to\mathbb{R}:[f]_{W^{\lambda, \theta_0}(0,T;H^{-1})}+\Vert f\Vert_{L^\infty(0,T;H)}\leq L_{1}(\epsilon)\right\},
	\end{equation*}
	then 
	\begin{equation*}
		\inf_{\mu\in(0,\mu_{T})}\mathbb{P}\big(\rho_{\mu}+\mu\partial_{t}u_{\mu}\in K_{1}^{\epsilon}\big)>1-\epsilon/4.
	\end{equation*}
	According to \eqref{end7}, we have that for every $p<(\theta+1)/\theta$
	\[\lim_{h\to 0}\Vert\tau_hf-f\Vert_{L^p(0,T-h;H^{-1})}=0,\ \ \ \ f \in\,K^\epsilon_1.\]
	Hence, in view of \cite[Theorem 6]{simon1986}, we have that $K^\e_1$ is relatively compact in $L^q(0,T,H^{-\delta})$, for every $q<\infty$ and $\delta>0$.
	
	Next, due to \eqref{est2}, we have
	\begin{equation*}
		\lim_{\mu\to0}\mathbb{E}\norm{\mu\partial_{t}u_{\mu}}_{L^{2}([0,T];H)}^{2}=0,
	\end{equation*}
	hence for every sequence $(\mu_{k})_{k\in\mathbb{N}}\subset (0,\mu_{T})$ converging to zero, there exists a compact $K_{2}^{\epsilon}$ in $L^{2}([0,T];H)$ such that
	\begin{equation*}
		\mathbb{P}\big(\mu_{k}\partial_{t}u_{\mu_{k}}\in K_{2}^{\epsilon}\big)>1-\epsilon/4,\ \ \ k\in\mathbb{N}.
	\end{equation*}
	Since $L^{2}([0,T];H)\subset L^{\theta_{0}}([0,T];H^{-\delta})$, for $\delta>0$, we have that $K_{2}^{\epsilon}$ is also compact in $L^{\theta_{0}}([0,T];H^{-\delta})$, which implies that $K_{1}^{\epsilon}+K_{2}^{\epsilon}$ is relatively compact in $L^{\theta_{0}}([0,T];H^{-\delta})$, and for every $k\in\mathbb{N}$,
	\begin{equation*}
		\mathbb{P}\big(\rho_{\mu_{k}}\in K_{1}^{\epsilon}+K_{2}^{\epsilon}\big)\geq 1-\epsilon/2.
	\end{equation*}
	Moreover, thanks to  estimate \eqref{est1} 	there exists $L_{2}(\epsilon)>0$ such that, if we define 
	\begin{equation*}
		K_{3}^{\epsilon}=\Big\{f:[0,T]\times\mathbb{R}\to\mathbb{R}:\norm{f}_{L^{\theta+1}([0,T];H)}\leq L_{2}(\epsilon)\Big\},
	\end{equation*}
	then 
	\begin{equation*}
		\inf_{\mu\in(0,\mu_{T})}\mathbb{P}\big(\rho_{\mu}\in K_{3}^{\epsilon}\big)\geq 1-\epsilon/4,
	\end{equation*}
	and thus
	\begin{equation*}
		\inf_{k\in\mathbb{N}}\mathbb{P}\Big(\rho_{\mu_{k}}\in \big(K_{1}^{\epsilon}+K_{2}^{\epsilon}\big)\cap K_{3}^{\epsilon}\Big)\geq 1-3\epsilon/4.
	\end{equation*}
	By using again \cite[Theorem 6]{simon1986}, $\big(K_{1}^{\epsilon}+K_{2}^{\epsilon}\big)\cap K_{3}^{\epsilon}$ is relatively compact in $L^{p}([0,T];H^{-\delta})$ for every $\delta>0$ and $p<\theta+1$. This implies that the family of probability measures $\big(\mathcal{L}(\rho_{\mu_{k}})\big)_{k\in\mathbb{N}}$ is tight in $L^{p}([0,T];H^{-\delta})$, for every $\delta>0$ and $p<\theta+1$.
	
	Now, according to  \cite[Theorem 1]{simon1986},	we have
	\begin{equation*}
		\lim_{h\to0}\sup_{f\in \big(K_{1}^{\epsilon}+K_{2}^{\epsilon}\big)\cap K_{3}^{\epsilon}}\norm{\tau_{h}f-f}_{L^{p}([0,T];H^{-\delta})}=0.
	\end{equation*}
	Furthermore, since 
	\begin{equation*}
		\sup_{\mu\in[0,\mu_{T}]}\mathbb{E}\norm{\rho_{\mu}}_{L^{2}([0,T];H^{1})}^{2}<\infty,
	\end{equation*}
	there exists $L_{3}(\epsilon)>0$ such that, if we define 
	\begin{equation*}
		K_{4}^{\epsilon}=\Big\{f:[0,T]\times\mathbb{R}\to\mathbb{R}\ :\ \norm{f}_{L^{2}([0,T];H^{1})}\leq L_{3}(\epsilon)\Big\},
	\end{equation*}
	then 
	\begin{equation*}
		\inf_{\mu\in(0,\mu_{T})}\mathbb{P}\big(\rho_{\mu}\in K_{4}^{\epsilon}\big)\geq 1-\epsilon/4.
	\end{equation*}
	Thus, if we define 
	\begin{equation*}
		K^{\epsilon}=\big(K_{1}^{\epsilon}+K_{2}^{\epsilon}\big)\cap K_{3}^{\epsilon}\cap K_{4}^{\epsilon},
	\end{equation*}
	we have
	\begin{equation*}
		\inf_{k\in\mathbb{N}}\mathbb{P}(\rho_{\mu_{k}}\in K^{\epsilon})\geq 1-\epsilon.
	\end{equation*}
	
	For every $a\in[0,1)$ and for any $\delta>0$, by interpolation  
	\begin{equation*}
		\norm{u}_{H^{a}}\leq C(a,\delta)\norm{u}_{H^{-\delta}}^{\frac{1-a}{1+\delta}}\norm{u}_{H^{1}}^{\frac{a+\delta}{1+\delta}}.
	\end{equation*}
	Thus, according  to \cite[Theorem 7]{simon1986}, we have $K^{\epsilon}$ is relatively compact in $L^{q}(0,T; H^{a})$, where $q=q(a,\delta,p)$ satisfies 
	\begin{equation*}
		\frac{1}{q}=\frac{1-a}{p(1+\delta)}+\frac{a+\delta}{2(1+\delta)},\ \ \ \delta>0,\ \ p<\theta+1.
	\end{equation*}
	This means that $K^{\epsilon}$ is relatively compact in $L^{q}(0,T; H^{a})$, for every $q<q(a)$, where 
	\begin{equation*}
		q(a)=\frac{2(\theta+1)}{2+(\theta-1)a},\ \ \ a\in[0,1),
	\end{equation*}
	so that
	$\big(\mathcal{L}(\rho_{\mu_{k}})\big)_{k\in\mathbb{N}}$ is tight in $X_{1}(a)$.
	This, together with the tightness of $\big(\mathcal{L}(\rho_{\mu_{k}})\big)_{k\in\mathbb{N}}$ in $X_2(\delta)$ proved above, completes the proof of part (1).

	Now let us assume that the condition (\ref{sigma_poly}) holds. Thanks to (\ref{est4}) we know that
	\begin{equation*}
		\lim_{\mu\to0}\mathbb{E}\norm{\mu\partial_{t}u_{\mu}}_{L^{\infty}([0,T];H)}=0.
	\end{equation*}
	Since $L^\infty(0,T;H)\subset L^q(0,T;H^{-\delta})$, for every $q<\infty$ and $\delta >0$, we can proceed as in the proof of part (1), and we have that $\big(\mathcal{L}(\rho_{\mu_{k}})\big)_{k\in\mathbb{N}}$ is tight in $L^{p}([0,T];H^{-\delta})$ for every $p<\infty$ and $\delta>0$.
	Finally, since 
	\begin{equation*}
		\sup_{\mu\in[0,\mu_{T}]}\mathbb{E}\norm{\rho_{\mu}}_{L^{2}([0,T];H^{1})}^{2}<\infty,
	\end{equation*}
	by using the same argument as in the proof of part (1), we have that for every $a\in[0,1)$, $\big(\mathcal{L}(\rho_{\mu_{k}})\big)_{k\in\mathbb{N}}$ is tight in $L^{q}([0,T];H^{a})$, where $q=q(a,\delta,p)$ satisfies
	\begin{equation*}
		\frac{1}{q}=\frac{1-a}{p(1+\delta)}+\frac{a+\delta}{2(1+\delta)},\ \ \ \delta>0,\ \ p<\infty.
	\end{equation*}
	This implies that $\big(\mathcal{L}(\rho_{\mu_{k}})\big)_{k\in\mathbb{N}}$ is tight in $L^{q}([0,T];H^{a})$, for every $q<2/a$, and the proof of part (2) follows.
	
\end{proof}

\subsection{The limiting problem}
Here we will prove the existence and uniqueness of solutions for the following equation
\begin{equation*}
	\left\{\begin{array}{l}
		\displaystyle{\gamma(u(t,x))\partial_t u(t,x)= \Delta u(t,x) +f(u(t,x)) -\frac{\gamma'(u(t,x))}{2\gamma^2(u(t,x))} \sum_{i=1}^\infty |[\sigma(u(t,\cdot)Qe_i](x)|^2}\\[18pt]
		\displaystyle{\ \ \ \ \ \ \ \ \ \ \ \ \ \ \ \ \ \ \ \ \ \ \ \ \ \ \ \ +\sigma(u(t,\cdot))\partial_t w^Q(t,x),}\\[18pt]
		\displaystyle{u(0,x)=u_0(x), \ \ \ \ \ \ \ u(t,0)=u(t,L)=0.}
	\end{array}\right.
\end{equation*}
To this purpose, we shall first study the following quasilinear parabolic equation
\begin{equation}\label{limit_para}
	\left\{
	\begin{array}{l}
		\displaystyle{\partial_{t}\rho(t,x)=\text{div}\big[b(\rho(t,x))\nabla\rho(t,x)\big]+f_{g}(x,\rho(t,x))+\sigma_{g}(\rho(t,\cdot))dw^{Q}(t,x)}\\[10pt]
		\displaystyle{\rho(0,x)=g(u_{0}(x))},\ \ \ \ \rho(t,0)\rho(t,L)=0.
	\end{array}\right.
\end{equation}	

\begin{Definition}
	An $(\mathcal{F}_{t})_{t\geq0}$ adapted process $\rho\in L^{2}(\Omega;L^{2}(0,T;H^{1}))$ is  a solution of equation \eqref{limit_para} if for every test function $\psi\in C^{\infty}_{0}([0,L])$
	\begin{equation}
		\begin{array}{l}
			\displaystyle{\inner{\rho(t),\psi}_{H}
				=\inner{g(u_{0}),\psi}_{H}+\int_{0}^{t}\inner{b(\rho(s))\nabla\rho(s),\nabla\psi}_{H}ds}\\[10pt]
			\displaystyle{\ \ \ \ \ \ \ \ \ \ \ \ \ \ \ \ \ \ \ \ \ \ \ \ +\int_{0}^{t}\inner{F_{g}(\rho(s)),\psi}_{H}ds+\int_{0}^{t}\inner{\sigma_{g}(\rho(s))dw^{Q}(s),\psi}_{H}.}
		\end{array}
	\end{equation}
\end{Definition}

\begin{thm}\label{well-posedness-para}
	Under Hypotheses \ref{Hypothesis1}, \ref{Hypothesis2} and \ref{Hypothesis3}, there is a unique solution $\rho$ to equation \eqref{limit_para} such that for every $p<\infty$	\begin{equation}\label{sm3000}
		\mathbb{E}\sup_{t\in[0,T]}\norm{\rho(t)}_{H}^{p}+\int_{0}^{T}\mathbb{E}\norm{\rho(t)}_{H^{1}}^{2}dt<\infty.  
	\end{equation}
\end{thm}

\begin{proof}	If we define the operator $\mathcal{S}: H^{1}\to H^{-1}$ by setting
	\begin{equation*}
		\mathcal{S}(\rho):=\text{div}\big[b(\rho)\nabla\rho\big]+F_{g}(\rho),\ \ \ \ \rho\in H^{1},
	\end{equation*}
	then equation \eqref{limit_para} can be written as
	\begin{equation}\label{limit_para_adj}
		\partial_{t}\rho(t,x)=\mathcal{S}(\rho(t,x))+\sigma_{g}(\rho(t,\cdot))dw^{Q}(t,x),
	\end{equation}
	with $\rho(0,\cdot)=g(u_{0}(\cdot))\in H^{1}$ and $\rho(t,0)=\rho(t,L)=0$.
		
	According to \cite[Theorem 1.1]{liu2010}, the well-posedness result follows once we prove that for all $\rho, \rho_{1}, \rho_{2}\in H^{1}$ the following conditions hold
	\begin{enumerate}
	\item[1.] {\em Hemicontinuity:}  The map $s\mapsto \inner{\mathcal{S}(\rho_{1}+s\rho_{2}),\rho}_{H^{-1},H^{1}}$ is continuous on $\mathbb{R}$;
	\item[2.] {\em Local monotonicity:}
	\begin{equation*}
		\begin{aligned}	
			&\inner{\mathcal{S}(\rho_{1})-\mathcal{S}(\rho_{2}),\rho_{1}-\rho_{2}}_{H^{-1},H^{1}}+\norm{\sigma_{g}(\rho_{1})-\sigma_{g}(\rho_{2})}_{\mathcal{L}_{2}(H_{Q},H)}^{2}\\
			&\ \ \ \ \ \ \ \ \ \ \ \ \ \ \ \leq -\frac{1}{\gamma_{1}}\norm{\rho_{1}-\rho_{2}}_{H^{1}}^{2}+c\big(1+\norm{\rho_{2}}_{H^{1}}^{2}\big)\norm{\rho_{1}-\rho_{2}}_{H}^{2}.
		\end{aligned}
	\end{equation*}
	\item[3.] {\em Coercivity:}
	\begin{equation*}
		\inner{\mathcal{S}(\rho),\rho}_{H^{-1},H^{1}}+\norm{\sigma_{g}(\rho)}_{\mathcal{L}_{2}(H_{Q},H)}^{2}\leq -\frac{1}{\gamma_{1}}\norm{\rho}_{H^{1}}^{2}-c_{2}\norm{\rho}_{L^{\theta+1}}^{\theta+1}+c\big(1+\norm{\rho}_{H}^{2}\big).
	\end{equation*}
	\item[4.] {\em Growth:}
	\begin{equation*}
		\norm{\mathcal{S}(\rho)}_{H^{-1}}^{2}\leq c\Big(1+\norm{\rho}_{H^{1}}^{2}\Big)\Big(1+\norm{\rho}_{H}^{2(\theta-1)}\Big).
	\end{equation*}

	\end{enumerate}
	In fact, the first property  is a consequence of the definition of $\mathcal{S}(\cdot)$. As for properties 2, 3 and 4, they follow from  the Lipschitz continuity of $\sigma_{g}$, properties \eqref{sm108} and \eqref{sm109} 
		and the fact that
	\begin{equation}\label{property3}
		\norm{F_{g}(\rho)}_{H^{-1}}\leq c\Big(1+\norm{\rho}_{L^{\theta}}^{\theta}\Big)\leq C\Big(1+\norm{\rho}_{H^{1}}\norm{\rho}_{H}^{\theta-1}\Big).
	\end{equation}

	Therefore, by applying Theorem 1.1 in \cite{liu2010} (with $V=H^{1}$, $H=H$, $\rho(\cdot)=c\norm{\cdot}_{H^{1}}^{2}, \alpha=2$ and $\beta=2(\theta-1)$), we can conclude that there exists a unique solution $\rho$ to equation \eqref{limit_para_adj} and \eqref{sm3000} holds.

\end{proof}

By proceeding as in the proof of \cite[Theorem 7.1]{CX}, the well-posedness of problem \eqref{limit_para} implies the well-posedness of problem \eqref{SPDE3}
\begin{cor}\label{well-posedness-limit}
	For every $T>0$ and $p<\infty$, and every $u_{0}\in H^{1}$, there exists a unique solution $u$ to the limiting problem \eqref{SPDE3} and 
	\begin{equation*}
		\mathbb{E}\sup_{t\in[0,T]}\norm{u(t)}_{H}^{p}+\int_{0}^{T}\mathbb{E}\norm{u(t)}_{H^{1}}^{2}dt<\infty.
	\end{equation*}
	\end{cor}

%%%%%%%%%%%%%%%%%%%%%%%%%%%%%%%%%%%%%

\subsection{Proof of Theorem \ref{sm}}
For every sequence $\{\mu_{k}\}_{k\in\mathbb{N}}\subset(0,\mu_{T})$ converging to $0$ as $k\to\infty$, we denote 
\begin{equation*}
	u_{k}:=u_{\mu_{k}}\ \ \ \text{and}\ \ \ \rho_{k}:=g(u_{k}),\ \ \ k\in\mathbb{N}.
\end{equation*}
In view of the first part of Proposition \ref{sm_tightness}, if we define 
\begin{equation*}
	X_{1}:=\bigcap_{0\leq a<1}X_{1}(a)\ \ \ \text{and}\ \ \ X_{2}:=\bigcap_{\delta>0}X_{2}(\delta),
\end{equation*}
we have that the family
\begin{equation*}
	\big\{\mathcal{L}(\rho_{k},\ \mu_k\partial_t u_{k})\big\}_{k \in\,\mathbb{N}} \subset \mathcal{P}\Big(\big(X_{1}\cap X_{2}\big)\times L^{2}(0,T;H)\Big)
\end{equation*}
is tight. 

We denote by $\rho$ a weak limit point for the sequence $\{\rho_k\}_{k \in\,\mathbb{N}}$ and we denote
\begin{equation*}
	\mathcal{K}:=\big(X_{1}\cap X_{2}\big)\times L^{2}(0,T;H)\times C([0,T],U),
\end{equation*}
where $U$ is the Hilbert space such that the embedding $H_{Q}\subset U$ is Hilbert-Schmidt.
According to the Skorokhod Theorem there exist  random variables
\begin{equation*}
	\mathcal{Y}=(\hat{\rho}, 0, \hat{w}^Q),\ \ \ \ \ \ \ \ 
	\mathcal{Y}_k=\left( \hat{\rho}_k,\hat{\theta}_k,\hat{w}^Q_k  \right),\ \ \ \ k \in\,\mathbb{N},
\end{equation*}
defined on a probability space $(\hat{\Omega}, \hat{\mathcal{F}}, \{\hat{\mathcal{F}}_t\}_{t \in\,[0,T]}, \hat{\mathbb{P}})$, such that
\begin{equation*}
	\mathcal{L}(\mathcal{Y})=\mathcal{L}(\rho, 0,w^Q),\ \ \ \ \ \ \ \ \ \ \mathcal{L}(\mathcal{Y}_k)=\mathcal{L}(\rho_{k},\ \mu_k\partial_t u_{k},w^Q),\ \ \ \ \ k \in\,\mathbb{N},
\end{equation*}
and such that
\begin{equation}\label{converge1}
	\lim_{k\to\infty}\mathcal{Y}_k=\mathcal{Y} \ \ \text{in}\ \ \mathcal{K},\ \ \ \ \hat{\mathbb{P}}\text{-a.s.}
\end{equation}
In particular, 
\begin{equation}\label{converge2}
	\lim_{k\to\infty}\Big(\norm{\hat{\rho}_{k}-\hat{\rho}}_{L^{p}([0,T];H)}+\norm{\hat{\rho}_{k}-\hat{\rho}}_{L^{q}([0,T];C([0,L]))}\Big)=0,\ \ \ \hat{\mathbb{P}}\text{-a.s.}
\end{equation}
for every $p<\theta+1$ and every $q<4(\theta+1)/(\theta+3)$.
By proceeding as in the proof of   \cite[Theorem 7.1]{CX}, thanks to Corollary \ref{well-posedness-limit},  in order to prove Theorem \ref{sm}, it is sufficient to show that $\hat{\rho}$ solves the parabolic equation \eqref{limit_para}.

For every $k\in\mathbb{N}$ and $\psi\in C^{\infty}_{0}([0,L])$, we have
\begin{equation}\label{weak_form}
	\begin{aligned}
		\inner{\hat{\rho}_{k}(t)+\hat{\theta}_{k}(t),\psi}_{H}
		&=\inner{g(u_{0})+\mu_{k}v_{0},\psi}_{H}-\int_{0}^{t}\inner{b(\hat{\rho}_{k}(s))\nabla\hat{\rho}_{k}(s),\nabla\psi}_{H}ds\\
		& +\int_{0}^{t}\inner{F_{g}(\hat{\rho}_{k}(s)),\psi}_{H}ds+\int_{0}^{t}\inner{\sigma_{g}(\hat{\rho}_{k}(s))d\hat{w}_{k}^{Q}(s),\psi}_{H},\ \ \hat{\mathbb{P}}\text{-a.s.}
	\end{aligned}
\end{equation}
Since $\hat{\rho}_{k}+\hat{\theta}_{k}$ converges to $\hat{\rho}$ in $L^{2}(0,T;H)$, $\hat{\mathbb{P}}$-a.s., we have
\begin{equation}\label{key1}
	\lim_{k\to\infty}\int_0^t\inner{\hat{\rho}_{k}(s)+\hat{\theta}_{k}(s),\psi}_{H}\,ds=\int_0^t\inner{\hat{\rho}(s),\psi}_{H}\,ds,\ \ \ \ \ t \in\,[0,T],\ \ \ \ \hat{\mathbb{P}}\text{-a.s.}
\end{equation}
As in the proof of \cite[Theorem 7.1]{CX}, we have that
\begin{equation}
\label{key2}
	\lim_{k\to\infty}\ \sup_{t \in\,[0,T]}\left|\int_{0}^{t}\inner{b(\hat{\rho}_{k}(s))\nabla\hat{\rho}_{k}(s),\nabla\psi}_{H}ds-\int_{0}^{t}\inner{b(\hat{\rho}(s))\nabla\hat{\rho}(s),\nabla\psi}_{H}ds\right|=0,\ \ \hat{\mathbb{P}}\text{-a.s.}
\end{equation}
Next, as \eqref{sm15-l1} extends to $F_g$, for each $k\in\mathbb{N}$,
\begin{equation*}
	\begin{aligned}
		&\Big\lvert \int_{0}^{t}\inner{F_{g}(\hat{\rho}_{k}(s))-F_{g}(\hat{\rho}(s)),\psi}_{H}\Big\rvert\\
		&\leq c\norm{\psi}_{H^{1}}\int_{0}^{t}\Big(1+\norm{\hat{\rho}_{k}(s)}_{L^{2(\theta-1)}}^{\theta-1}+\norm{\hat{\rho}(s)}_{L^{2(\theta-1)}}^{\theta-1}\Big)\norm{\hat{\rho}_{k}(s)-\hat{\rho}(s)}_{H}ds\\
		&\leq c\norm{\psi}_{H^{1}}\Big(1+\norm{\hat{\rho}_{k}}_{L^{q(\theta-1)}([0,T];C([0,L])}^{\theta-1}+\norm{\hat{\rho}}_{L^{q(\theta-1)}([0,T];C([0,L])}^{\theta-1}\Big)\norm{\hat{\rho}_{k}-\hat{\rho}}_{L^{p}([0,T];H),}
	\end{aligned}
\end{equation*}
for any $p$, $q$ satisfying 
\begin{equation}\label{condition_pq}
	\frac{4(\theta+1)}{7+2\theta-\theta^{2}}<p<\theta+1,\ \ \ \ \frac{1}{p}+\frac{1}{q}=1.
\end{equation}
One can  check that for any pair $p,q>1$ satisfying \eqref{condition_pq}, it holds $q(\theta-1)<4(\theta+1)/(\theta+3)$. Then thanks to \eqref{converge2}, we have
\begin{equation}
\label{key3}
	\lim_{k\to\infty}
\sup_{t \in\,[0,T]}\left|\int_{0}^{t}\inner{F_{g}(\hat{\rho}_{k}(s)),\psi}_{H}ds-\int_{0}^{t}\inner{F_{g}(\hat{\rho}(s)),\psi}_{H}ds\right|=0,\ \ \ \hat{\mathbb{P}}\text{-a.s.}
\end{equation}
Finally,  since
\begin{equation*}
	\lim_{k\to\infty}\sup_{t\in[0,T]}\norm{\hat{w}_{k}^{Q}(t)-\hat{w}^{Q}(t)}_{U}=0,\ \ \ \mathbb{P}\text{-a.s.}
\end{equation*}
and
\begin{equation*}
	\lim_{k\to\infty}\norm{\hat{\rho}_{k}-\hat{\rho}}_{L^{2}([0,T];H)}=0,\ \ \ \mathbb{P}\text{-a.s.}
\end{equation*}
with the uniform estimate
\begin{equation*}
	\sup_{k\in\mathbb{N}}\ \mathbb{E}\sup_{t\in[0,T]}\norm{\hat{\rho}_{k}(t)}_{H}^{2}<\infty,
\end{equation*}
by \cite[Corollary 4.5]{gyongy} we have that
\begin{equation}\label{key4}
	\lim_{k\to\infty}\sup_{t \in\,[0,T]}\left|\int_{0}^{t}\inner{\sigma_{g}(\hat{\rho}_{k}(s))d\hat{w}_{k}^{Q}(s),\psi}_{H}-\int_{0}^{t}\inner{\sigma_{g}(\hat{\rho}(s))d\hat{w}^{Q}(s),\psi}_{H}\right|=0,\ \ \ \ \text{in probability}.
\end{equation}

Therefore, combining \eqref{key1}--\eqref{key4}, if we integrate  with respect to time  both sides of equation \eqref{weak_form} and take the limit as $k\to\infty$, it follows that for every $\psi\in C^{\infty}_{0}([0,L])$ and $t \in\,[0,T]$,
\begin{equation*}
	\begin{aligned}
	\int_0^t\inner{\hat{\rho}(s),\psi}_{H}\,ds
		&=\int_0^t\left[\inner{g(u_{0}),\psi}_{H}-\int_{0}^{s}\inner{b(\hat{\rho}(r))\nabla\hat{\rho}(r),\nabla\psi}_{H}dr\right.\\
		& \ \ \left.+\int_{0}^{s}\inner{F_{g}(\hat{\rho}(r)),\psi}_{H}dr+\int_{0}^{s}\inner{\sigma_{g}(\hat{\rho}(r))d\hat{w}^{Q}(r),\psi}_{H}\right]\,ds,\ \ \hat{\mathbb{P}}\text{-a.s.}
	\end{aligned}
\end{equation*}
Due to the arbitrariness of $t \in\,[0,T]$, this means that $\hat{\rho}\in L^{2}(\Omega;X_{1}\cap X_{2}\cap L^{2}(0,T;H^{1}))$ solves equation \eqref{limit_para} with initial data $u_{0}$, and the first part of the theorem is proved. 

We omit the proof of the second part as it is analogous to the one we have just seen.

%%%%%%%%%%%%%%%%%%%%%%%%%%%%%%%%%%%%%%%%%%%

\end{document}